\documentclass[twoside,11pt]{article}

\usepackage[abbrvbib, preprint]{jmlr2e}

\usepackage{hyperref}
\hypersetup{colorlinks = true,
            linkcolor = blue,
            urlcolor  = blue,
            citecolor = blue,
            anchorcolor = blue}

\usepackage{url}

\usepackage{enumitem}
\usepackage[utf8]{inputenc} 
\usepackage[T1]{fontenc}    
\usepackage{url}            
\usepackage{booktabs}       
\usepackage{amsfonts}       
\usepackage{nicefrac}       
\usepackage{microtype}      
\usepackage{algorithmic}
\usepackage{algorithm}
\usepackage{amsmath}
\usepackage{bm}
\usepackage{mathtools}
\usepackage{mathrsfs}
\usepackage{subcaption}
\usepackage{caption}
\usepackage[nolist]{acronym}
\usepackage{bbm}
\usepackage[dvipsnames]{xcolor}
\usepackage{wrapfig}
\usepackage{siunitx}
\usepackage{multirow}
\usepackage{todonotes}

\usepackage[symbol]{footmisc}
\usepackage{float}
\usepackage{cleveref}
\usepackage{caption}

\usepackage{tikz}

\DeclareMathOperator{\tr}{\rm{tr}}
\DeclareMathOperator{\Var}{Var}
\DeclareMathOperator{\Cov}{Cov}

\newcommand{\rd}{{\mathrm d}}

\newcommand{\T}{^\top}

\newcommand{\calL}{{\cal L}}

\newcommand{\calN}{{\cal N}}
\newcommand{\calM}{{\cal M}}

\newcommand{\argmax}{\operatornamewithlimits{argmax}}
\newcommand{\argmin}{\operatornamewithlimits{argmin}}
\newtheorem{assumption}{Assumption}

\newcommand{\Pb}{\mathbb{P}}

\newcommand{\KL}[2]{\mathcal{D}_{KL}\left({#1}\parallel {#2}\right)}

\newcommand{\D}[3]{\mathcal{D}_{#1}\left({#2}\parallel {#3}\right)}

\newcommand{\Rb}{\mathbb{R}}

\newcommand{\Eb}{\mathbb{E}}

\newcommand{\ExP}[2]{\Eb_{{#1}}{\left[#2\right]}}

\begin{acronym}
\acro{CEM}{Cross Entropy Method}
\acro{GASS}{Gradient-based Adaptive Stochastic Search}
\acro{MPPI}{Model Predictive Path Integral}
\acro{SA}{Stochastic Approximation}
\acro{DNN}{Deep Neural Network}
\acro{DP}{Dynamic Programming}
\acro{FBSDE}{Forward-Backward Stochastic Differential Equation}
\acro{FSDE}{Forward Stochastic Differential Equation}
\acro{BSDE}{Backward Stochastic Differential Equation}
\acro{LSTM}{Long-Short Term Memory}
\acro{FC}{Fully Connected}
\acro{DDP}{Differential Dynamic Programming}
\acro{HJB}{Hamilton-Jacobi-Bellman}
\acro{PDE}{Partial Differential Equation}
\acro{PI}{Path Integral}
\acro{NN}{Neural Network}
\acro{GPs}{Gaussian Processes}
\acro{SOC}{Stochastic Optimal Control}
\acro{RL}{Reinforcement Learning}
\acro{MPOC}{Model Predictive Optimal Control}
\acro{IL}{Imitation Learning}
\acro{RNN}{Recurrent Neural Network}
\acro{FNN}{Feed-forward Neural Network}
\acro{Single FNN}{Single Feed-forward Neural Network}
\acro{DL}{Deep Learning}
\acro{SGD}{Stochastic Gradient Descent}
\acro{SDE}{Stochastic Differential Equation}
\acro{2BSDE}{second-order BSDE}
\acro{2FBSDE}{second-order FBSDE}
\acro{OCE}{Optimized Certainty Equivalents}
\acro{CVaR}{Conditional Value-at-Risk}
\acro{VaR}{Value-at-Risk}
\acro{DRO}{Distributionally Robust Optimization}
\acro{PMP}{Pontryagin Maximum Principle}
\acro{VI}{Variational Inference}
\acro{ELBO}{Evidence Lower BOund}
\acro{EM}{Expectation Maximization}
\acro{SVGD}{Stein Variational Gradient Descent}
\acro{RKHS}{Reproducing Kernel Hilbert Space}
\acro{SME}{Stochastic Modified Equation}
\acro{BO}{Bayesian Optimization}
\acro{SS}{Stochastic Search}
\acro{ARA}{Absolute Risk Aversion}
\acro{QGASS}{Quantum \ac{GASS}}
\acro{MPC}{Model Predictive Control}
\acro{GPU}{Graphic Processing Unit}
\acro{VO}{Variational Optimization}
\acro{TVI}{Tsallis Variational Inference}
\acro{ADMM}{Alternating Direction
Method of Multipliers}
\acro{MC}{Monte Carlo}
\acro{SM}{Supplementary Material}
\end{acronym}

\DeclarePairedDelimiter\abs{\lvert}{\rvert}%

\usepackage{anyfontsize}

\firstpageno{1}

\begin{document}

\title{Sampling-Based Optimization for Multi-Agent Model Predictive Control}

\author{\name Ziyi Wang \email ZiyiWang@gatech.edu
       \AND
       Augustinos D. Saravanos \email asaravanos@gatech.edu
       \AND
      Hassan Almubarak \email halmubarak@gatech.edu
      \AND
      Oswin So \email
      oswinso@gatech.edu
      \AND
      \name Evangelos A.  Theodorou \email evangelos.theodorou@gatech.edu \\
      \addr Autonomous Control and Decision Systems Lab \\
      \addr School of Aerospace Engineering\\
      Georgia Institute of Technology, GA 30318, USA}

\editor{Kevin Murphy and Bernhard Sch{\"o}lkopf}

\maketitle

\begin{abstract}
We systematically review the Variational Optimization, Variational Inference and Stochastic Search perspectives on sampling-based dynamic optimization and discuss their connections to state-of-the-art optimizers and Stochastic Optimal Control (SOC) theory. A general convergence and sample complexity analysis on the three perspectives is provided through the unifying Stochastic Search perspective. We then extend these frameworks to their distributed versions for multi-agent control by combining them with consensus Alternating Direction Method of Multipliers (ADMM) to decouple the full problem into local neighborhood-level ones that can be solved in parallel. Model Predictive Control (MPC) algorithms are then developed based on these frameworks, leading to fully decentralized sampling-based dynamic optimizers. The capabilities of the proposed algorithms framework are demonstrated on multiple complex multi-agent tasks for vehicle and quadcopter systems in simulation. The results compare different distributed sampling-based optimizers and their centralized counterparts using unimodal Gaussian, mixture of Gaussians, and stein variational policies. The scalability of the proposed distributed algorithms is demonstrated on a 196-vehicle scenario where a direct application of centralized sampling-based methods is shown to be prohibitive.
\end{abstract}



\section{Introduction}
\label{sec: introduction}

Stochastic optimization is a family of methods that relies on sampling of random variables to solve optimization problems. For deterministic problems, the decision variables are generated from a sampling distribution, which is updated through a weighted average of the samples. For stochastic problems, stochasticity can also arise in the problem formulation from the objective function or constraints, and realizations of problem scenarios are sampled. There are two main ingredients for sampling-based optimizers: the sampling distribution and evaluation criterion. The sampling distribution for decision variables determines the region to be explored, while the evaluation criterion measures the importance of each sample as a function of its cost and is often designed through heuristics and metaheuristics. 

The choice of sampling distribution and evaluation criterion depend on the theoretical foundation behind each algorithm. Simulated annealing \citep{kirkpatrick1983optimization} is an algorithm inspired by thermodynamics and samples one candidate solution around the current iterate through random perturbation. The probability of transition, which is a function of the energy/cost of the candidate solution, determines whether to use the candidate solution for the next iteration. Another class of methods are the Genetic algorithms \cite{mitchell1998introduction}, which mimic the evolutionary process, encode the decision variable with genomes, and sample a population in the genome space, while the cost of each sample is used to evolve the population. The \ac{CEM} was initially developed for rare event estimation and has later been applied to optimization problems. In \ac{CEM}, samples are generated from distributions of the exponential family, and an elite subset of samples with cost below a threshold is used to update the distribution. Recently, \cite{zhou2014gradient} proposed the general \ac{SS} algorithm with the update law derived through gradient descent of the transformed cost function with respect to the parameters of the sampling distribution from the exponential family. Algorithms such as \ac{CEM} can be recovered as a special case of \ac{SS} with specific choices of the cost transform.


Sampling-based methods have also been widely used in dynamic optimization, where the optimization problem is defined for a dynamical system over a certain period of time. Similar to static optimization, sampling-based dynamic optimizers are derived from different disciplines. Genetic algorithms \citep{michalewicz1990genetic}, \ac{CEM} \citep{kobilarov2012cross}, and \ac{SS} \citep{wang2019information} have been directly applied to dynamic optimization. On the other hand, \ac{MPPI} \citep{williams2016aggressive} is a popular dynamic optimizer derived through variational optimization with connections to the well-known Hamilton-Jacobi-Bellman theory in stochastic optimal control. Variational Inference \ac{MPC} \citep{okada2020variational} is another variational scheme that reformulates control problems as a Bayesian inference ones. These methods have enjoyed much success recently with applications to autonomy \citep{williams2016aggressive, williams2018best}, manipulation \citep{nagabandi2020deep} and even social dynamics \citep{wang2017variational, wang2022opinion}. They have also been extended to handle constrained \citep{boutselis2020constrained} and risk sensitive \citep{wang2021adaptive} problem formulations. Compared to their gradient-based counterparts, sampling-based optimizers have been shown to be more robust and less prone to undesirable local minima \citep{williams2017model}. In addition, thanks to their capabilities to handle nonsmooth cost and dynamics, they are widely applicable to a variety of systems and tasks \citep{williams2016aggressive, nagabandi2020deep, evans2021stochastic}. On the other hand, the sampling nature of this type of optimization frameworks makes it challenging to scale to high dimensional problems such as large scale multi-agent control. This is due to the fact that as the number of agents increases, it becomes increasingly hard to sufficiently explore the state space with the sampled trajectories.

The issue of scalability can be surpassed by deploying distributed optimization techiniques. One such technique is \ac{ADMM} \citep{boyd2011distributed}, which relies on decomposing high-dimensional problems into smaller subproblems that can be solved separately. Consensus \ac{ADMM} is a particular ADMM variation which relies on enforcing a consensus between the sub-problems and is readily applicable to multi-agent control problems with separable structure. Recent works such as \citep{saravanos2021distributed, halsted2021survey, tang2021fast, pereira2022decentralized, saravanos2022distributed} have utilized consensus \ac{ADMM} to propose decentralized schemes that are scalable and computationally efficient for multi-agent control. ADMM has also been used in the \ac{MPC} setting for deterministic multi-agent coordination and collision avoidance \citep{lu2014separable, rey2018fully, cheng2021admm}. Nevertheless, to the best of our knowledge, prior \ac{ADMM}-based methods for multi-agent control have not considered sampling-based approaches. 

In this paper, we systematically review different sampling-based dynamic optimization methods, provide a unified analysis, and propose a distributed sampling-based dynamic optimization framework for multi-agent control. Our contributions are as follows:

\begin{enumerate}
\item We review three different perspectives on sampling-based dynamic optimization, namely Stochastic Search, Variational Optimization and Variational Inference. For each perspective, we consider the unimodal Gaussian, Gaussian mixture, and Stein variational policy distributions.
The connections between the perspectives and state-of-the-art algorithms are discussed in Section \ref{sec: sampling_opt}. 
\item We provide a unified convergence and sample complexity analysis of the different perspectives in Section \ref{sec: analysis}. The sample complexity analysis also allows for the comparison of the variance of approximaton error between perspectives.
\item We propose a distributed \ac{MPC} framework that uses consensus \ac{ADMM} to scale up general sampling-based dynamic optimizers in Section \ref{sec: distributed_opt}. The distributed optimization formulation decouples the multi-agent problem into neighborhood-level ones, allowing for full decentralization. Furthermore, the resulting \ac{MPC} algorithm can adapt the neighborhoods at each timestep.
\item The analysis and the proposed distributed framework is tested in simulation in Section \ref{sec: results}. The variance comparison analysis is verified on a single-agent point-mass system. Subsequently, a comparison of performance for different optimizers and policy distributions are carried out on 4 different multi-vehicle and multi-quadcopter tasks. Finally, we demonstrate the increased scalability and superior performance of the distributed framework against the centralized one through a series of experiments that scale up to 196 vehicles. 
\end{enumerate}

\section{Different Perspectives in Sampling-based Dynamic Optimization}
\label{sec: sampling_opt}

In this section, we review three perspectives on sampling-based nonlinear dynamic optimization, namely \ac{VO}, \ac{VI}, and \acl{SS}. Consider the discrete-time dynamic optimization problem over time horizon $T$:
\begin{equation} \label{eq: dyn_opt}
\begin{split}
U^* & = \argmin_U J(\tau)\\
\text{s.t.}\hspace{3mm} \ & x_{t+1} = F(x_t, u_t)\\
& x_0 = \bar{x}_0,
\end{split}
\end{equation}
where $\tau = \{X, U\}$ with $X=\{x_1, x_2, \cdots, x_T\}\in\Rb^{n_x\times T}$, $U=\{u_0, u_1, \cdots, u_{T-1}\}\in\Rb^{n_u\times T}$, and $\bar{x}_0$ is the initial state. Upper case letters denote a trajectory of variables over the same time horizon unless otherwise noted. We overload the definition of $J$ for the state cost, control cost, and the total cost of the problem, as well as the cost at each timestep. Sampling-based optimization assumes that the controls can be sampled from a distribution $u_t \sim q(u_t;\theta_t)$ with distribution parameters $\theta_t\in\Rb^{n_p}$. Different perspectives provide different interpretations  on the same optimization problem, leading to different algorithmic behaviors and application scenarios.

\subsection{Kullback-Leibler Variational Optimization Perspective}
\label{sec: kl_vo}
The \ac{VO} approach \citep{williams2017information} leverages the \textit{free energy} of the system defined as
\begin{equation} \label{eq: free_energy}
\mathcal{F}(J(X)) = -\lambda\log\left(\mathbb{E}_{p(\tau)}\left[\exp\big(-\frac{1}{\lambda}J(X)\big)\right]\right) = \min_q \mathbb{E}_{q(\tau)}[J(X)] + \lambda \KL{q(\tau)}{p(\tau)},
\end{equation}
where $\lambda$ is the inverse temperature and $\KL{\cdot}{\cdot}$ is the KL divergence. Here $p(\tau)$ and $q(\tau)$ define the probability distributions of trajectories for uncontrolled and controlled dynamics respectively. Since the dynamics is deterministic, we have $q(\tau)=q(U)=\prod_{t=0}^{T-1} q(u_t;\theta_t)$. The optimization on the right hand side admits a closed form solution in the form of Gibbs distribution as follows
\begin{equation} \label{eq: vo_update}
q^*(U) = \frac{\exp(-\frac{1}{\lambda} J(X))p(U)}{\Eb_{p(U)} [\exp(-\frac{1}{\lambda} J(X))]}.
\end{equation}
An iterative scheme can be derived by minimizing the KL divergence between the optimal distribution and the controlled distribution, i.e. $\KL{q^*(U)}{q(U;\Theta)}$, leading to an update law similar to that of the \ac{MPPI} algorithm \citep{williams2016aggressive}. A full derivation is included in \Cref{appendix: KL_VO}.

\subsubsection{Renyi Variational Optimization}
\label{sec: renyi_vo}
While the KL \ac{VO} approach results in a policy update law of the form in \Cref{tab: perspective_comparison}, the update law does not incorporate the normalization of the cost function in the \ac{MPPI} update, which was added as a necessary heuristic to robustify computation \citep{theodorou2010generalized, williams2018information}. Here we show that when using the $\alpha-$Renyi divergence \citep{renyi1961measures} instead of the KL one, the normalization of the cost function arises in a natural way by using the proper choice of the $\alpha$ parameter and making it adaptive. Starting from the Renyi divergence definition and following the steps of the \ac{VO} approach we have: 
\begin{align}
    \Theta^{*} &= \argmin {\cal{D}}_{\alpha}\big(q^{*}(U) || q(U;\Theta) \big) = \frac{1}{\alpha-1}\log\Eb_{q^*}\left[\left(\frac{q^*}{q}\right)^{\alpha-1}\right],
    \label{eq:Renyi}
\end{align}
where $ \alpha > 0  $ and $ \lim_{\alpha \to 0}{\cal{D}}_{\alpha}   =  {\cal{D}}_{KL}$. By taking the gradient of $ {\cal{D}}_{\alpha} $ and following an importance sampling procedure and change of measure step, one can derive the iterative scheme:
\begin{equation}
    \Theta^{k+1}  =  \frac{\Eb_{q^k}\left[ \left( \exp\left(-\frac{1}{\lambda} J(\tau)\right) \right)^{\alpha} U \right]}{\Eb_{q^k}\left[\left( \exp\left(-\frac{1}{\lambda} J(\tau)\right) \right)^{\alpha} \right]}.
\end{equation}
The derivation can be found in \Cref{appendix: Renyi_VO}. By defining the $ \alpha $ parameter in the Renyi divergence as $ \alpha = \frac{1}{\max J - \min J}$ one can derive the MPPI update law with the normalization of the cost $ J $ to  range $[0,1]$.

\subsubsection{Connections to Hamilton-Jacobi-Bellman Theory}
\label{sec: connection_hjb}
In this section, we demonstrate the connection between the \ac{VO} approach and \ac{SOC} theory. In particular, we show equivalence between the free energy and the value function. To facilitate the analysis, we work on a continuous time \ac{SOC} formulation. Consider the following standard \ac{SOC} problem
\begin{equation}\label{eq: soc_problem}
\begin{split}
u^* &= \argmin_u \underbrace{\Eb\left[\int_0^T (Q(x(t)) + \frac{1}{2}u\T Ru)\rd t + Q_T(x(T))\right]}_{J(x(t),u)}\\
\rd x &= A(x) \rd t + B(x)(u\rd t + \sqrt{\lambda}\rd w),
\end{split}
\end{equation}
where $\lambda$ is the inverse temperature as defined earlier and $w(t):\Rb\rightarrow \Rb^{n_w}$ is a multivariate Brownian motion. Despite the quadratic control cost and affine control assumption, many problems in \ac{SOC} and robotics can be represented in this form. The dynamics can be discretized as
\begin{equation}\label{eq: discretized_dynamics}
x_{t+1} = x_t + A(x_t) \delta t + B(x_t)(u_t \delta t + \sqrt{\lambda} \delta w_t) = F(x_t, u_t + \delta w_t).
\end{equation}

We can define the value function as
\begin{equation}
V(x(t),t) = \inf_u J(x(t),u), \quad V(x(T), T) = Q_T(x(T)).
\end{equation}
Using stochastic Bellman's principle \citep{yong1999stochastic}, the value function's solution can be shown to satisfy the Hamilton-Jacobi-Bellman equation
\begin{equation}\label{eq: hjb}
\partial_t V + \frac{1}{2}\lambda\tr\left(\nabla_{xx}VBB\T\right) + \nabla_x V\T A + Q - \frac{1}{2}\nabla_xV\T B R^{-1} B\T\nabla_x V = 0.
\end{equation}
\Cref{eq: hjb} is a semi-linear PDE due to the last term being nonlinear with respect to $\nabla_x V$. We choose $R=I$ and introduce $\psi(x,t)=\exp(-\frac{1}{\lambda}V(x,t))$ to get
\begin{equation}\label{eq: desirability}
\partial_t \psi + \frac{1}{2}\lambda\tr\left(\nabla_{xx}\psi B B\T\right) + \nabla_x \psi\T A - \frac{1}{\lambda} \psi Q = 0,
\end{equation}
with
\begin{equation*}
\begin{split}
\partial_t V &= -\lambda \frac{1}{\psi} \partial_t \psi\\
\nabla_x V &= -\lambda \frac{1}{\psi} \nabla_x \psi\\
\nabla_{xx} V &= -\lambda \frac{1}{\psi^2} \nabla_x \psi \nabla_x \psi\T - \lambda \frac{1}{\psi} \nabla_{xx} \psi.
\end{split}
\end{equation*}
The transformed PDE in \cref{eq: desirability} is linear with respect to $\psi$. The Feynman-Kac lemma \citep{friedman2010stochastic} states that for the linear PDE in \cref{eq: desirability}, the solution can be written as
\begin{equation}\label{eq: feyman_kac_solution}
\psi(x, t) = \Eb\left[\exp\left(\int_0^T -\frac{1}{\lambda}(Q(x(t)) + \frac{1}{2}u\T u) \rd t\right)\right]=\Eb\left[\exp\left(-\frac{1}{\lambda}J\right)\right],
\end{equation}
where the second equality denotes time discretization. From the definition of free energy and $\psi$, we can observe the equivalence between free energy and value function.

\subsection{Variational Inference Perspective}
\label{sec: vi}

The \ac{VI} perspective \citep{okada2020variational} takes a Bayesian approach to the dynamic optimization problem. The goal of the inference problem is to sample from the posterior distribution
\begin{equation} \label{eq: vi_bayes}
\begin{split}
p(\tau|o=1) &= \frac{p(o=1|\tau)p(\tau)}{p(o=1)}\\
&= \frac{p(o=1|\tau)p(U)}{p(o=1)},
\end{split}
\end{equation}
where $o\in\{0,1\}$ is a dummy variable that indicates optimality with $o=1$. The \ac{VI} objective can be formulated by minimizing the KL divergence between the controlled distribution and the posterior distribution, i.e. $\KL{q(\tau)}{p(\tau|o=1)}$. Using \cref{eq: vi_bayes}, the objective can be transformed to
\begin{equation} \label{eq: vi_obj}
q^*(U) = \argmin_{q} -\Eb_{q(U)}[\log p(o=1|\tau)] + \KL{q(U)}{p(U)},
\end{equation}
where $p(o=1|\tau)$ is the optimality likelihood and can be parameterized by a non-increasing function of the trajectory cost as $p(o=1|\tau)\coloneqq l(J(\tau))$. In this paper, we choose $l(x)=\exp(-\frac{1}{\lambda}x)$ such that $\log p(o=1|\tau)=-\lambda^{-1}J(\tau)$ and the objective becomes
\begin{equation} \label{eq: vi_obj_final}
q^*(U) = \argmin_{q} \lambda^{-1}\Eb_{q(U)}[J(\tau)] + \KL{q(U)}{p(U)}.
\end{equation}
Note that the original objective \eqref{eq: vi_obj} is a generalization of the \ac{VO} objective \eqref{eq: free_energy}.

\subsubsection{Varational Inference Using Tsallis Divergence}
\label{sec: tsallis_vi}

\begin{figure}[t]
    \centering
    \includegraphics[width=0.9\linewidth]{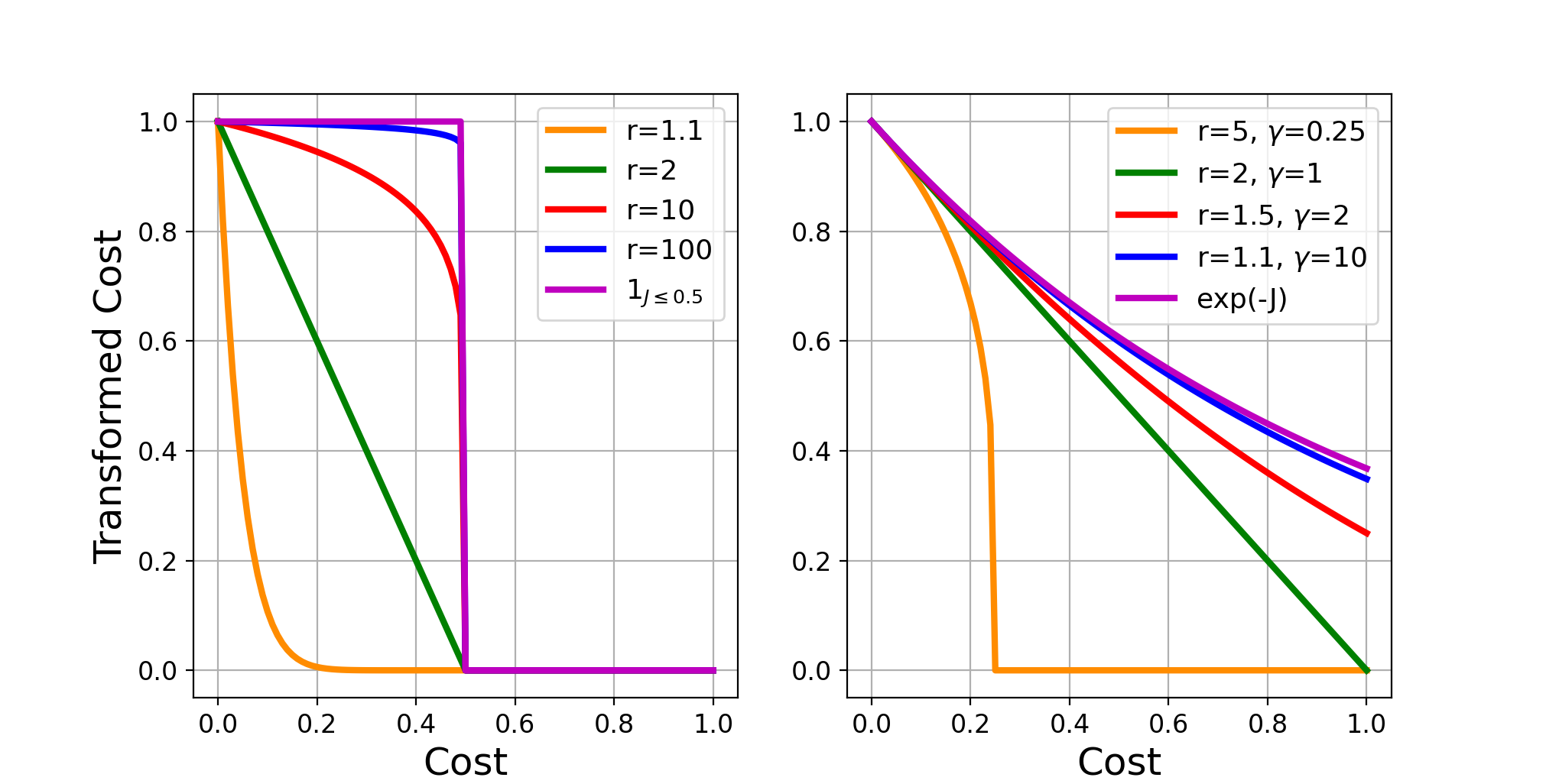}
    \caption{Reparameterized $\exp_r$ \eqref{eq:reparameterization}. \textit{Left}: with $\gamma=0.5$, $\exp_r$ approaches the indicator function $\bm{1}_{\{J\leq 0.5\}}$ as $r\rightarrow \infty$. \textit{Right}: $\exp_r$ approaches the negative exponential function $\exp(-J)$ with appropriate choices of $\gamma$ as $r\rightarrow 1$.}
    \label{fig:exp_r}
\end{figure}

The \ac{TVI} approach \citep{wang2021variational} further generalizes this formulation by using the Tsallis divergence \citep{tsallis1988possible} instead of the KL divergence as the regularizer
\begin{equation}\label{eq: tvo_obj}
q^*(U) = \argmin_q \lambda^{-1}\Eb_{q(U)}[J(\tau)] + \D{r}{q(U)}{p(U)},
\end{equation}
where
\begin{equation*}
\D{r}{q(x)}{p(x)} \coloneqq  \Eb_q\left[\log_r\frac{q(x)}{p(x)}\right], \hspace{3mm} \log_r(x) = \frac{x^{r-1}-1}{r-1} , \hspace{3mm}
\exp_r(x) = (1+(r-1)x)_+^{\frac{1}{r-1}},
\end{equation*}
with $(\cdot)^+=\max(0, \cdot)$. Here the deformed exponential and logarithmic functions $\exp_r$ and $\log_r$ are generalizations of the standard $\exp$ and $\log$ for $r>1$. We have $\exp_r\rightarrow\exp$ and $\log_r\rightarrow\log$ with $r\rightarrow 1$. Similarly, the KL divergence is recovered from Tsallis divergence $\mathcal{D}_r$ as $r\rightarrow 1$. The optimal control distribution can be computed in closed-form (detailed in \Cref{appendix: Tsallis_VI} as
\begin{equation}\label{eq:tsallis_vi_soc:dist_update}
q^*(U)
=\frac{\exp_r \left( -\lambda^{-1} J\right)p(U)} {\Eb_{p(U)}\left[\exp_r \left( -\lambda^{-1} J\right)\right]}.
\end{equation}

\subsubsection{Reparameterized Tsallis Variational Inference}
\label{sec: reparameterized_tsallis}
We can reparameterize the $\exp_r$ term in \cref{eq:tsallis_vi_soc:dist_update} as $\lambda=(r-1)\gamma$ and get
\begin{equation}
\label{eq:reparameterization}
\exp_r(-\lambda^{-1} J) = \left(1-\frac{J}{\gamma}\right)_+^{\frac{1}{r-1}}
= \begin{cases}
\exp\left(\frac{1}{r-1}\log \left(1-\frac{J}{\gamma}\right)\right), \, J<\gamma \\
0, \hspace{9.6 em} J \geq \gamma
\end{cases},
\end{equation}
where $\gamma$ is now the threshold beyond which the optimality weight is set to 0.
The reparameterization adjusts the original parameter $\lambda$ at every iteration to maintain the same threshold $\gamma$, which is a more intuitive parameter.
In addition, the reparameterized framework is easier to tune and achieves better performance than the original formulation \citep{wang2021variational}. Therefore, we focus our analysis and simulations only on the reparameterized version hereon after. Note that in practice, it is easier to define an \textit{elite fraction}, which adjusts $\gamma$ based on the scale of the costs, instead of using $\gamma$ for easier tuning. Using elite fraction achieves a similar effect to the normalization heuristic of \ac{MPPI}.

\Cref{fig:exp_r} illustrates the cost function transformation corresponding to different $r$ and $\gamma$ values.
For $J < \gamma$, as $r \rightarrow \infty$, $\left(1-\frac{J}{\gamma}\right)_+^{\frac{1}{r-1}} \rightarrow 1$.
Hence, $\left(1-\frac{J}{\gamma}\right)_+^{\frac{1}{r-1}}$ converges pointwise to the step function $\bm{1}_{ \{J \leq \gamma\} }$ with $r \to \infty$.
On the other hand, since $\exp_r\rightarrow\exp$ as $r\rightarrow 1$, $\left(1-\frac{J}{\gamma}\right)_+^{\frac{1}{r-1}}\rightarrow \exp(-\lambda J)$ with $r\rightarrow 1$ and $\gamma=\frac{\lambda}{r-1}$.

\subsection{Stochastic Search Perspective}
\label{sec: ss}
The \ac{SS} perspective \citep{boutselis2020constrained, wang2021adaptive} is grounded on stochastic optimization theory and optimizes with respect to the relaxed objective $\Eb[J(\tau|\Theta)]$ under the sampled control distribution assumption. To facilitate the algorithmic development the objective is transformed into 
\begin{equation} \label{eq: ss_obj}
\Theta^* = \arg\max_\Theta \ln\left(\mathbb{E}\left[S\left(- J(\tau)\right)\right]\right),
\end{equation}
where $S$ is a non-decreasing shape function. Common shape functions include: 1) the exponential function, $S(x;\lambda)=\exp(\frac{1}{\lambda} x)$, which leads to the same update law as the KL \ac{VO} approach and the unnormalized \ac{MPPI} algorithm; 2) the deformed exponential function, $S(x;\lambda, r)=\exp_r(\frac{1}{\lambda} x)$, which leads to the \ac{TVI} update law; 3) the sigmoid function, $S(x; \kappa, \varphi)=\tfrac{1}{1+\exp(-\kappa(x-\varphi))}$, where $\varphi$ is the ($1-\rho$)-quantile, which results in an update law similar to \ac{CEM} but with soft elite threshold $\rho$. Note that the optimization is performed with respect to the distribution parameters $\Theta$ as opposed to the distribution itself in the variational schemes. Assume that the control distribution is of the exponential family such that its pdf is of the form
\begin{equation}\label{eq: exponential_family_pdf}
q(u_t;\theta_t) = h(u_t)\exp\Big(\eta(\theta_t)^\mathrm{T}T(u_t)-A(\theta_t)\Big),
\end{equation}
with $\eta$ being the natural parameters of the distribution and $T(u)$ being the sufficient statistics of $u$. The update law can be derived by performing gradient ascent on the objective \eqref{eq: ss_obj} with respect to the natural parameters as
\begin{equation} \label{eq: ss_natural_param_update}
\eta_t^{k+1} 
= \eta_t^k + \alpha^k\Bigg(\frac{\mathbb{E}\Big[S(-J(\tau)) (T(u_t) - \Eb[T(u_t)]) \Big]}{\mathbb{E} \Big[S(-J(\tau))\Big]}\Bigg),
\end{equation}
and then transforming it to the distribution parameters. We refer the readers to \Cref{appendix: SS} for the detailed derivation.

\subsection{Policy Distribution}
\label{sec: policy}
For the \ac{VO} and \ac{VI} approaches, it is computationally inefficient to sample from $q^*(U)$ through \eqref{eq: vo_update} and \eqref{eq:tsallis_vi_soc:dist_update}. Instead, $q^*(U)$ can be approximated iteratively through a tractable distribution $q(U)$ by minimizing the KL divergence between $q$ and $q^*$
\begin{equation} \label{eq:update_laws:obj_fn}
q^{k+1}(U) = \argmin_{q} \KL{q^*(U)}{q(U)} .
\end{equation}
However, because one can only evaluate $q^*(U)$ at a finite number of points $\{ U^{(m)} \}_{m=1}^M$, we instead approximate $q^*(U)$ by the empirical distribution $\hat{q}^*(U)$ with weights $w^{(m)}$:
\begin{align}
    \hat{q}^*(U) = \sum_{m=1}^M w^{(m)} \bm{1}_{\{U = U^{(m)} \}}, \\
    w^{(m)} = \frac{q^*(U^{(m)})}{\sum_{m'=1}^M q^*(U^{(m')}) } .
\end{align}

We now solve \cref{eq:update_laws:obj_fn} for 3 different policy classes: unimodal Gaussian, Gaussian mixture. and a nonparametric policy corresponding to \ac{SVGD} \citep{lambert2020stein, liu2016stein}. The detailed derivation for each policy distribution is in \Cref{appendix: policy}.

\subsubsection{Unimodal Gaussian} For a unimodal Gaussian policy distribution with parameters $\Theta\coloneqq\{ (\mu_t, \Sigma_t) \}_{t=0}^{T-1}$, the update laws for the $(k+1)$-th iteration take the form of
\begin{subequations}
\begin{align} 
\mu^{k+1}_t &=
    \sum_{m=1}^M w^{(m)} u_t^{(m)}, \label{eq:mean_update} \\
\Sigma^{k+1}_t &=
    \sum_{m=1}^M w^{(m)} (u_t^{(m)}-\mu_t^{k+1})(u_t^{(m)}-\mu_t^{k+1})\T \label{eq:var_update} .
\end{align}
\end{subequations}

\subsubsection{Gaussian Mixture} Alternatively, the policy distribution can be an $L$-mode mixture of Gaussian distributions with parameters $\Theta \coloneqq \{ \Theta_l \}_{l=1}^{L}$ with $\Theta_l \coloneqq (\phi_l, \{ \mu_{l, t}, \Sigma_{l, t} \}_{t=0}^{T-1})$, where $\phi_l$ is the mixture weight for the $l$-th component.
Although it is not possible to directly solve \cref{eq:update_laws:obj_fn} in this case, we draw from \ac{EM} to derive an iterative update scheme for the $(k+1)$-th iteration with the form
\begin{subequations}
\begin{align}
\phi_l^{k+1} &= \frac{N_l}{\sum_{l'=1}^L N_{l'}},\\
\mu_{l,t}^{k+1} &= \frac{1}{N_l}\sum_{m=1}^M \eta_l(u_t^{(m)}) w^{(m)} u_t^{(m)},\\
\Sigma_{l,t}^{k+1} &= \frac{1}{N_l}\sum_{m=1}^M \eta_l(u_t^{(m)}) w^{(m)} (u^{(m)}_t-\mu_{l,t}^{k+1})(u^{(m)}_t-\mu_{l,t}^{k+1})\T,
\end{align}
\end{subequations}
where
\begin{align*}
\eta_l(u_t^{(m)}) &= \frac{\phi_l^k\calN(u_t^{(m)}|\mu_{l,t}^k, \Sigma_{l,t}^k)}{\sum_{l'=1}^L \phi_{l'}^k \calN(u_t^{(m)} | \mu_{l',t}^k, \Sigma_{l',t}^k)},\\
N_{l} &= \sum_{m=1}^M \sum_{t=0}^{T-1} \eta_l(u_t^{(m)}) w^{(m)}.
\end{align*}

\subsubsection{Stein Variational Policy}
The policy can also be a non-parametric distribution approximated by a set of particles $\Theta\coloneqq\{\Theta_l\}_{l=1}^L$ for some parametrized policy $\hat{q}(U; \Theta)$. In \cite{lambert2020stein}, $\hat{q}$ is taken to be a unimodal Gaussian with fixed variance with $\Theta$ corresponding to the mean.
The update law of each Stein particle for the $(k+1)$-th iteration has the form
\begin{subequations}
\begin{align}
    \Theta_l^{k+1}
        &= \Theta_l^{k} + \epsilon \hat{\phi}^*(\Theta_l^k), \\
    \hat{\phi}^*(\Theta)
        &= \sum_{l=1}^L \hat{k}(\Theta_{l}, \Theta) G(\Theta_l) + \nabla_{\Theta_l} \hat{k}(\Theta_l, \Theta), \\
    G(\Theta_l) &= \frac{\sum_{s=1}^S w^{(l, s)} \nabla_\Theta \log \hat{\pi}(U^{(m)}, \Theta_{l})}{\sum_{s=1}^S w^{(l, s)}} ,
\end{align}
\end{subequations}
where $M$ is chosen such that $M = LS$, $\hat{k}$ is a kernel function, and $w^{(l, s)} = w^{(m+L(l-1))}$, where $S$ denotes the number of rollouts for each of the $L$ particles.
As noted in \cite{zhuo2018message}, \ac{SVGD} becomes less effective as the dimensionality of the particles increases due to the inverse relationship between the repulsion force in the update law and the dimensionality.
Hence, we follow \cite{lambert2020stein} in choosing a sum of local kernel functions as our choice of $\hat{k}$.

\bgroup
\begin{table*}[t]
    \centering
    \caption{Comparison of the minimization objective and the update law for a unimodal Gaussian policy with fixed variance between different dynamic optimization approaches.}
    \begin{tabular}{@{} l l l @{}}
    \toprule
    \multicolumn{1}{c}{Approach} & \multicolumn{1}{c}{Objective (Minimize)} & \multicolumn{1}{c}{Update Law} \\
    \midrule
     Tsallis VI & $\mathbb{E}[J] + \lambda \D{r}{q}{p}$ & $\theta_t^{k+1} = \sum_{m=1}^M\frac{\exp_r \left(-\lambda^{-1}J^m\right) u^m_t }{\sum_{m'=1}^M \exp_r \left(-\lambda^{-1}J^{m'}\right)}$ \\
    \addlinespace[0.3em]
    Reparameterized TVI & $\mathbb{E}[J] + \lambda \D{r}{q}{p}$ & $\theta^{k+1}_t = \sum_{m=1}^M\frac{\left(1-\frac{J^m}{\gamma}\right)^{\frac{1}{r-1}}_+ u^m_t}{\sum_{m'=1}^M \left(1-\frac{J^{m'}}{\gamma}\right)^{\frac{1}{r-1}}_+}$\\
    \addlinespace[0.3em]
    VO/Unnormalized MPPI  & $\mathbb{E}[J] + \lambda \KL{q}{p}$ & $\theta_t^{k+1} = \sum_{m=1}^M \frac{\exp \left(-\lambda^{-1}J^m\right) u^m_t }{\sum_{m'=1}^M \exp \left(-\lambda^{-1}J^{m'}\right)}$ \\
    \addlinespace[0.3em]
    SS & $\Eb[S(-J)]$ & $\theta_t^{k+1} = \theta_t^k + \sum_{m=1}^M\frac{S(J^m)  (u^m_t - \frac{1}{M}\sum_{m''=1}^M( u^{m''}_t)) }{\sum_{m'=1}^M S(J^{m'})}$ \\
    \addlinespace[0.5em]
    CEM & $\Eb[J]$ & $\theta_t^{k+1}=\alpha \theta_t^k + (1-\alpha) \sum_{m=1}^M \frac{\bm{1}_{ \{J^m\leq \gamma\} } u_t^m }{\sum_{m'=1}^M \bm{1}_{ \{J^{m'}\leq \gamma\} }}$\\
    \bottomrule
    \end{tabular}
    \label{tab: perspective_comparison}
\end{table*}
\egroup

\subsection{Connections and Comparisons between Perspectives}
\label{sec: connections}
In this section, we compare between the different perspectives derived and against state-of-the-art sampling-based dynamic optimization algorithms such as \ac{MPPI} and \ac{CEM}. A comparison of problem formulations and update laws for different approaches with a unimodal Gaussian policy is demonstrated in \Cref{tab: perspective_comparison}.

\ac{CEM} \citep{de2005tutorial} is widely used in reinforcement learning and optimal control problems \citep{mannor2003cross, kobilarov2012cross}.
The objective of \ac{CEM} is to minimize the expected cost $\Eb [J]$. The policy update law for \ac{CEM} is based on a heuristic indicator function $\bm{1}_{\{J\leq \gamma\}}$ that averages the samples whose cost is below the elite threshold $\gamma$. In practice, a elite fraction is used instead of the unnormalized elite threshold.


It can be observed that the \ac{VO} and \ac{VI} perspectives share a similar objective function with the expected cost/negative log optimality likelihood and a regularization term. Both perspectives optimize directly with respect to the policy distribution, and the form of optimal policy distribution depends on the regularization term. From \Cref{sec: tsallis_vi} and \Cref{sec: reparameterized_tsallis}, we know that the objective of KL \ac{VO} can be recovered from \ac{TVI} by taking $r\rightarrow 1$, which results in an update law similar to that of \ac{MPPI} with an unnormalized cost. Normalization can be incorporated through the use of Renyi divergence as shown in \Cref{sec: renyi_vo}. On the other hand, the \ac{CEM} update law is a special case of the reparameterized \ac{TVI} by taking $r\rightarrow \infty$ with $\gamma$ corresponding to the unnormalized elite fraction.

\ac{SS} is a general stochastic optimization framework for policy distributions from the exponential family, where optimization is performed with respect to the policy distribution parameters via gradient descent. The form of the update is dictated by the choice of the shape function $S(\cdot)$ that transforms the cost function. As shown in \Cref{tab: perspective_comparison}, the update law of \ac{VO} and \ac{TVI} perspectives, as well as the ones of \ac{MPPI} and \ac{CEM}, can all be recovered from the \ac{SS} update. \ac{MPPI}/\ac{VO} corresponds to $S(x;\lambda)=\exp(\frac{1}{\lambda} x)$. \ac{CEM} update can be derived by taking $S(x; \kappa, \varphi)=(x-x_{\text{lb}})\tfrac{1}{1+\exp(-\kappa(x-\varphi))}$. \ac{TVI} update law and its reparameterized version are recovered with $S(x;\lambda, r)=\exp_r(\frac{1}{\lambda} x)$ and $S(x;\gamma, r)=\left(1+\frac{x}{\gamma}\right)_+^{\frac{1}{r-1}}$ respectively.

Despite the connections between perspectives, there exist unique characteristics associated with each perspective as a result of their theoretical foundation. From \Cref{sec: connection_hjb}, we know that the \ac{VO} approach bridges variational optimization and optimal control theory, grounding the method on well-established concepts. The \ac{VI} perspective offers flexibility in algorithmic design with the choice of regularizing divergence and optimality likelihood function. Both variational approaches optimize directly with respect to the distribution, allowing a wide selection of policy distributions compared to the more restricting exponential family for \ac{SS}. On the other hand, \ac{SS} is a general framework that recovers the update law of all other approaches. It provides a unified platform for the theoretical analysis of different perspectives in \Cref{sec: analysis}. Its optimization formulation also allows for the use of many popular optimization techniques for improved performance, such as accelerated gradient descent.

\section{Theoretical Analysis on Convergence and Sample Complexity}
\label{sec: analysis}

\subsection{Convergence Analysis}
\label{sec: convergence}
In this section, we provide the convergence guarantees as well as the rate of convergence for the sampling-based dynamic optimization framework. In \Cref{sec: connections}, it was shown that the update laws of all different perspectives can be derived from the \ac{SS} perspective. Therefore, the analysis is performed on the \ac{SS} formulation, while it can be applied to other frameworks via specific choices of shape function $S(\cdot)$. Consider the \ac{SS} dynamic optimization problem formulation
\begin{equation}
\Theta^* = \arg\max_\Theta \ln\left(\mathbb{E}\left[S\left(- J\right)\right]\right),
\end{equation}
with the corresponding parameter update law at each time step as
\begin{equation}
\theta^{k+1}_t = \theta_t^k + \alpha^k\Eb\left[\frac{S(-J)(T(u_t^k)-\Eb[T(u_t^k)])}{\Eb\left[S(-J)\right]}\right].
\end{equation}
The step size $\alpha^k$ is introduced to facilitate convergence. The expectation can be approximated with \ac{MC} sampling. We now follow the steps in \cite{zhou2014gradient}, which shows convergence results for \ac{SS} in static optimization, to analyze the convergence for dynamic optimization.

To start off, we can write the update laws in the form of a generalized Robins-Monro algorithm as
\begin{align}
\label{eq:robin_monro}
\theta^{k+1}_t = \theta^k_t + \alpha^k_t[D(\theta^k_t) + \epsilon^k_t]
\end{align}
where $D(\theta_t^k)$ represents the true update and $\epsilon_t^k$, which behaves like a noise term, denotes the difference between the \ac{MC} approximation, $\tilde{D}(\theta_t^k)$, and the true update. Now we make a few standard assumptions:
\begin{assumption}\label{assump: convergence}
\begin{enumerate}[label=(\alph*)]
    \item The step size sequence $\{\alpha^k\}$ satisfies $\alpha^k>0$ for all $k$, $\alpha^k \searrow 0$ as $k\rightarrow \infty$, and $\sum_k^\infty \alpha^k = \infty$.
    \item The sample size for approximation is $M^k = M^0 k^\zeta$ where $\zeta > 0$. In addition, $\{\alpha^k\}$ and $\{N^k\}$ jointly satisfy $\alpha^k/\sqrt{M^k} = O(k^{-\beta})$ for some constant $\beta > 1$.
    \item The sufficient statistic function $x\mapsto T(x)$ is bounded.
\end{enumerate}
\end{assumption}

With these assumptions, we can use Theorem 1 in \cite{zhou2014gradient} to show that $\{\theta_t^k\}$ converges w.p.1 to either a limiting set or a unique equilibrium point $\theta_t^*$ of the problem.

We proceed with the assumption that $\{\theta_t^k\}$ converges to a unique equilibrium point $\theta_t^*$. Define $J_D(\theta)=\nabla_\theta D(\theta)$ as the Hessian matrix of the objective function. Assume that $J_D$ is continuous and symmetric negative definite in a neighborhood of $\theta^*_t$ and select a step size sequence of $\alpha^k=\alpha^0/k^\alpha$ with $\alpha_0>0$ and $0<\alpha<1$, and a polinomially increasing sample size sequence $M^k=M^0k^\zeta$ with $M^0>1$ and $\zeta>0$. We can rewrite \eqref{eq:robin_monro} as
\begin{equation}
\delta^{k+1}_t = \delta^k_t + \alpha^0 k^{-\alpha} D(\theta_t^k) + \alpha^0 k^{-\alpha} \epsilon_t^k
\end{equation}
by subtracting $\theta_t^*$ on both sides and setting $\delta_t^k=\theta_t^k-\theta_t^*$. Using a first order Taylor expansion of $D(\theta_t^k)$ around $\theta^*$ and the fact that $D(\theta^*)=0$, we get
\begin{equation}
\delta_t^{k+1} = \delta_t^k + \alpha^0 k^{-\alpha} J_D(\tilde{\theta}_t^k)\delta_t^k + \alpha^0 k^{-\alpha} \epsilon^k_t
\end{equation}
where $\tilde{\theta}_t^k$ lies between $\theta_t^k$ and $\theta^*_t$. We can rewrite this as
\begin{equation}
\delta_t^{k+1}=(I-k^{-\alpha}\Gamma^k_t)\delta_t^k + k^{-(\alpha+\tau)/2}\alpha^0 W_t^k
\end{equation}
where
\begin{align*}
\Gamma_t^k &= -\alpha^0 J_D (\tilde{\theta}_t^k)\\
W_t^k &= k^{(\tau-\alpha)/2} \epsilon_t^k
\end{align*}
for some positive constant $\tau$. With the previous assumptions, we know that $\Gamma_t^k\rightarrow \Gamma_t$ with $\Gamma_t=-\alpha^0 J_D(\theta_t^*)$ since $\theta_t^k$ converges to $\theta_t^*$. If we let $\tau>\alpha$ and $M^k=M^0 k^{\tau-\alpha}$, then there exists a positive semi-definite matrix $\Sigma$ such that $\lim_{k\rightarrow \infty} \Eb[W^k_t W^{k\top}_t]=\Sigma$ w.p.1 and $\lim_{k\rightarrow\infty} \Eb[\bm{1}_{\|W_t^k\|^2\geq rk^\alpha}\|W_t^k\|^2]=0, \ \forall r>0$. These conditions indicate that the stochasticity has finite variance and is bounded above.
With this we can apply Theorem 2.2 in \cite{fabian1968asymptotic} and obtain
\begin{equation}
k^{\frac{\tau}{2}}(\theta_t^k-\theta_t^*)\stackrel{\text{dist}}{\longrightarrow} \calN(0, QMQ\T)
\end{equation}
where $Q$ is an orthogonal matrix such that $Q\T(-\alpha^0 J_D(\theta_t^*))Q=\Lambda$ with $\Lambda$ being a diagonal matrix and the $(i,j)$-th entry of $M$ is given by $M_{(i,j)}=((\alpha^0)^2 Q\T\Sigma Q)_{(i,j)}(\Lambda_{(i,i)}+\Lambda_{(j,j)})^{-1}$.

\subsection{Sample Complexity}
\label{sec: complexity}
In this section, we investigate the sampling complexity for the expectation approximation of the update law via \ac{MC} sampling. For notational simplicity, the controls and their corresponding policy parameters are assumed to be 1-dimensional. The results can be extended to the multivariate case straightforwardly by analyzing a single dimension at a time. Consider a generic update law of the form
\begin{equation} \label{eq: policy_update_true}
\Theta^{k+1} = \Theta^k + \frac{\Eb\left[S(-J)\phi(U)\right]}{\Eb\left[S(-J)\right]},
\end{equation}
where $S(\cdot)$ denotes the shape function and $\phi(U)=(T(U)-\Eb[T(U)])$. The \ac{MC} approximation is performed through
\begin{equation} \label{eq: policy_update_approx}
\hat{\Theta}^{k+1} = \Theta^k + \frac{\sum^M_{m=1}\left(S(-J^m)\phi(U^m)\right)}{\sum^M_{m'=1}\left(S(-J^{m'})\right)}.
\end{equation}

Following the sampling complexity results in \cite{yoon2022sampling} for \ac{MPPI}, we make the standard, easy to satisfy assumptions below
\begin{assumption}\label{assump: sample_complexity}
\begin{enumerate}[label=(\alph*)]
\item The error bound of \ac{MC} approximation $\epsilon_1$ is smaller than the normalization term $\Eb[S(-J)]$. i.e. $\epsilon_1 < \Eb[S(-J)]$
\item The cost function is non-negative $J\geq 0$.
\end{enumerate}
\end{assumption}
Under \Cref{assump: sample_complexity}, we can derive the following theorem.
\begin{theorem}\label{theorem: complexity}
The error between the true policy update \eqref{eq: policy_update_true} and its \ac{MC} approximation can be bounded as:
\begin{equation} \label{eq: error_bound}
\theta^{k+1}_t - \frac{E_2\epsilon_1}{E_1-\epsilon_1}-\left(1-\frac{\epsilon_1}{E_1-\epsilon_1}\right)\epsilon_2\leq \hat{\theta}^{k+1}_t \leq \theta^{k+1}_t + \frac{E_2\epsilon_1}{E_1-\epsilon_1}+\left(1+\frac{\epsilon_1}{E_1-\epsilon_1}\right)\epsilon_2,
\end{equation}
where
\begin{subequations}
\begin{align}\label{eq: hoefding}
\Pb\left(|\hat{E}_1 - E_1|\geq \epsilon_1\right)&\leq \rho_1= \exp\left(-\frac{2\epsilon_1^2}{M}\right)\\
\Pb\left(|\hat{E}_2 - E_2|\geq \epsilon_2\right)&\leq\rho_2= \frac{\psi(T(u_t))}{M\epsilon_2^2(\Eb[S(-J)]^2)}, \label{eq: error_var}
 \end{align}
\end{subequations}
and
\begin{align*}
E_1&=\Eb[S(-J)],\ \hat{E}_1=\sum^M_{m=1}\left(S(-J^m)\right)\\
E_2&=\Eb[S(-J)\phi(u_t)]/E_1,\  \hat{E}_2=\sum^M_{m=1}\left(S(-J^m)\phi(u_t^m)\right)/\hat{E}_1\\
\psi(T(u_t))&=\sqrt{M_4-4M_1M_3+6M_2M_1^2-3M_1^4}+M_2-M_1^2.
\end{align*}
Here $M_i=\Eb[T(u_t)^i]$ denotes the $i$-th moment of $T(u_t)$.
\end{theorem}
\begin{proof}
We first write the update law for a single timestep in a concise form as
\begin{equation}
\theta^{k+1}_t = \theta^k_t + \left(\frac{\Eb[S(-J)]}{\Eb[S(-J)]}\right)\frac{\Eb[S(-J)\phi(u_t)]}{\Eb[S(-J)]}=\theta^k_t + AB,
\end{equation}
where $A=\frac{\Eb[S(-J)]}{\Eb[S(-J)]}$ and $B=\frac{\Eb[S(-J)\phi(u_t)]}{\Eb[S(-J)]}$. Similarly, we write the approximation as
\begin{equation}
\hat{\theta}_t^{k+1} = \theta^k_t + \left(\frac{\Eb[S(-J)]}{\sum_{m'=1}^M( S(-J^m))}\right)\frac{\sum^M_{m=1}\left(S(-J^m)\phi(u_t^m)\right)}{\Eb[S(-J)]} = \theta^k_t + \hat{A}\hat{B},
\end{equation}
with $\hat{A}=\frac{\Eb[S(-J)]}{\sum_{m'=1}^M( S(-J^m))}$ and $\hat{B}=\frac{\sum^M_{m=1}\left(S(-J^m)\phi(u_t^m)\right)}{\Eb[S(-J)]}$. The approximation error is $\theta_t^k - \hat{\theta}^k_t=AB - \hat{A}\hat{B}$. For choices of shape function $S(\cdot)\in[0,1]$, we get \eqref{eq: hoefding} from Hoeffding's inequality. Note that all shape functions in this paper are bounded in $[0,1]$. With this we can bound $\hat{A}$ around $A=1$ as
\begin{equation}
1-\frac{\epsilon_1}{E_1-\epsilon_1}\leq \hat{A} \leq 1+\frac{\epsilon_1}{E_1-\epsilon_1}.
\end{equation}
\Cref{assump: sample_complexity}(a) indicates that both limits of the bound for $\hat{A}$ are positive. Therefore, the parameter error bound is upper bounded by the product of the bound for $\hat{A}$ and $\hat{B}$, with $\hat{B}$ as $B-\epsilon_2\leq \hat{B}\leq B+\epsilon_2 $. The product can be computed as
\begin{align*}
\left(1-\frac{\epsilon_1}{E_1-\epsilon_1}\right)\left(B-\epsilon_2\right)\leq &\hat{A}\hat{B}\leq \left(1+\frac{\epsilon_1}{E_1-\epsilon_1}\right)\left(B+\epsilon_2\right)\\
\Rightarrow \theta^{k+1}_t - \frac{B\epsilon_1}{E_1-\epsilon_1}-\left(1-\frac{\epsilon_1}{E_1-\epsilon_1}\right)\epsilon_2\leq &\hat{\theta}^{k+1}_t \leq \theta^{k+1}_t + \frac{B\epsilon_1}{E_1-\epsilon_1}+\left(1+\frac{\epsilon_1}{E_1-\epsilon_1}\right)\epsilon_2.
\end{align*}
\end{proof}
To obtain the error bound $\epsilon_2$, we first derive a bound $\Var\left[\frac{S(-J)\phi(u_t)}{\Eb[S(-J)]}\right]$ in the following lemma.
\begin{lemma}\label{lemma: var}
The bound on the variance of weighted policy parameter depends on $\Eb[S(-J)]$ and moments of $T(u_t)$ as
\begin{equation} \label{eq: var_bound}
\Var\left(\frac{S(-J)\phi(u_t)}{\Eb[S(-J)]}\right)=\frac{\Var\left(S(-J)\phi(u_t)\right)}{\Eb[S(-J)]^2} \leq \frac{\psi(T(u_t))}{\Eb[S(-J)]^2}.
\end{equation}
\end{lemma}
\begin{proof}
Using the definition of variance and covariance we get
\begin{align*}
\Var\left(S(-J)\phi(u_t)\right)&=\Cov(S(-J)^2,\phi(u_t)^2) + \Eb[S(-J)^2]\Eb[\phi(u_t)^2]-\Eb[S(-J)\phi(u_t)]^2\\
&\leq \Cov(S(-J)^2,\phi(u_t)^2) + \Eb[S(-J)^2]\Eb[\phi(u_t)^2]\\
&\leq \left(\Var(S(-J)^2)\Var(\phi(u_t)^2)\right)^{\frac{1}{2}} + \Eb[S(-J)^2]\Eb[\phi(u_t)^2]\\
&\leq \left(\Var(\phi(u_t)^2)\right)^{\frac{1}{2}} + \Eb[\phi(u_t)^2]\\
&= \psi(T(u_t))
\end{align*}
where the second inequality results from the Cauchy Schwartz inequality $\Cov(X,Y)\leq\sqrt{\Var(X)\Var(Y)}$ and the last inequality results from $\Eb[S(-J)^2]\leq 1$ and $\Var(S(-J)^2)\leq 1$, which is derived from the property of bounded random variable
\begin{equation*}
\Var(S(-J))\leq (1-\Eb[S(-J)])\Eb[S(-J)]\leq \Eb[S(-J)]\leq 1.
\end{equation*}
\end{proof}
With \eqref{eq: var_bound}, \cref{eq: error_var} can be derived through Chebyshev's inequality and the proof for \cref{theorem: complexity} is completed.

\Cref{theorem: complexity} provides a probabilistic bound on the \ac{MC} approximated parameter update $\hat{\theta} ^{k+1}_t$ from the desired one $\theta^{k+1}_t$. For cases where the expected transformed cost $\Eb[S(-J)]$ and moments of policy parameter's sufficient statistics $M_i$ can be explicitly calculated, the risks $\rho_1, \rho_2$ can be analyzed based on the number of samples $M$ and error bounds $\epsilon_1,\epsilon_2$. The risks decrease as the number of samples increase with $\rho_1=O(e^{-M})$ and $\rho_2=O(1/M)$.

\section{Distributed Sampling-based Dynamic Optimization}
\label{sec: distributed_opt}
In this section, we derive the distributed sampling-based dynamic optimization framework. 
\subsection{Formulation}
\label{sec: dist_sampling_derivation}
Consider a system of $N$ agents. For a representative $i$-th agent, the nonlinear discrete-time dynamics can be written as $x^i_{t+1} = F^i(x^i_t, u^i_t)$. 
The dynamics for the entire collection of agents can be represented in the compact form of $x_{t+1} = F(x_t, u_t)$.  We use letters without superscript to represent the collection of all agents (e.g. $x_t=[x^1_t; x^2_t; \cdots, x^N_t]$, $u_t=[u^1_t; u^2_t; \cdots; u^N_t]$ and $F(x_t, u_t)=[F^1(x^1_t, u^1_t); \cdots; F^N(x^N_t; u^N_t)]$).  
The multi-agent control problem can be formulated as
\begin{equation}
\begin{split} \label{eq: control_problem}
& U^* = \argmin_U \sum_{i=1}^N \sum_{t=0}^T J^i(x^i_t, u^i_t)\\
\textit{s.t.} \quad & x^i_{t+1} = F^i(x^i_t, u^i_t), \ \forall i=1,\cdots, N, \forall t = 0,\cdots,T-1\\
& g^i(x^i_t)\geq 0, \ \forall i=1,\cdots,N, \ \forall t = 1,\cdots,T \\
& h^{ij}(x^i_t, x^j_t) \geq 0, \ \forall i,j=1,\cdots,N, i\neq j, \ \forall t = 1,\cdots,T,
\end{split}
\end{equation}
where $g^i$ and $h^{ij}$ denote single-agent and inter-agent constraints, respectively. In multi-robot problems, the former could represent obstacle avoidance constraints and the latter collision avoidance constraints between different robots.

The control problem can be tackled by directly optimizing \eqref{eq: control_problem} with respect to all agents. This corresponds to the \textit{centralized} approach since a central unit collects information from all agents and computes their corresponding control policies. One critical issue with this approach is its scalability. As the number of agents increases, the dimensionality of the multi-agent problem increases as well, which prohibits directly employing sampling-based methods.  

To tackle this issue, we exploit the separable structure of the problem and proposed a \textit{decentralized} approach. To achieve this, the control of each agent is computed locally while satisfying collision constraints with respect to agents that are \textit{close enough} to the ego agent. To derive the decentralized approach, we first define $\calN^i_t$ as the neighborhood of agent $i$. There are two ways of defining the neighborhood: 1) distance-based neighborhood, $\calN^i_t=\{i=1,\cdots,N\backslash \{i\} : \|x^i_t-x^j_t\|_2\leq \epsilon\}$, which is characterized by a norm ball around the ego agent; 2) fixed-size neighborhood, where the set contains a fixed number of agents closest to the ego agent. In addition, we define $\calM^i_t$ as the collection of agents $j$ with agent $i$ in their neighborhood $\calM^i_t=\{j=1,\cdots,N\backslash \{i\} : i\in\calN^j_t\}$. 

\textbf{Global Consensus:} Define $\{x^{ij}\}$ and $\{u^{ij}\}$ as agent $i$'s perception of the neighbor agent $j$'s state and control. We can now augment the states and controls of each agent with $\tilde{x}^i_t=[x^i_t;\{x^{ij}_t\}_{j\in\calN^i_t}]$ and $\tilde{u}^i_t=[u^i_t;\{u^{ij}_t\}_{j\in\calN^i_t}]$. To ensure consensus between agents, we introduce global state and control variables $y=[y^1; y^2; \cdots; y^N], z=[z^1;z^2;\cdots;z^N]$ whose components should be equal to all local variables referring to the same agent. We can now formulate the optimization problem as
\begin{equation}
\begin{split}
 \{\tilde{U}^i\}_{i=1,\cdots,N}&=\argmin \sum_{i=1}^N \sum_{t=0}^T J^i(x^i_t,u^i_t)\\
\textit{s.t.} \quad & \tilde{x}^i_{t+1}=\tilde{F}^i(\tilde{x}^i_t,\tilde{u}^i_t); \hspace{3mm} g^i(x^i_t)\geq 0 \\
& \tilde{h}^i(\tilde{x}^i_t) \geq 0; \hspace{2mm} \tilde{x}^i_t = \tilde{y}^i_t; \hspace{2mm} \tilde{u}^i_t = \tilde{z}^i_t,
\end{split}
\end{equation}
where $\tilde{F}^i=[F^i;\{F^j\}_{j\in\calN^i_t}]$, $\tilde{h}_i(\tilde{x}^i_t) = [\{h_{ij}(x^i_t, x^{ij}_t)\}_{\calN^i_t}]$, $\tilde{y}^i_t=[y^i_t;\{y^j_t\}_{j\in\calN^i_t}]$ and   $\tilde{z}^i_t=[z^i_t;\{z^j_t\}_{j\in\calN^i_t}]$. With this formulation, we can follow a consensus \ac{ADMM} approach to solve the problem in a decentralized manner. For each agent $i$, the optimization procedure is as follows:
\begin{enumerate}
\item Solve
\begin{equation} \label{eq: primal_problem}
\begin{split}
\tilde{U}^i = &\argmin \ J^i(X^i, U^i) + L(\tilde{X}^i, \tilde{U}^i, \tilde{Y}^i, \tilde{Z}^i,\Xi^i, \Gamma^i)\\
\textit{s.t.} \quad & \tilde{x}^i_{t+1} = \tilde{F}^i(\tilde{x}^i_t, \tilde{u}^i_t); \hspace{2mm} g^i(x^i_t)\geq 0; \hspace{2mm} \tilde{h}^i(\tilde{x}^i_t) \geq 0,
\end{split}
\end{equation}
where
\begin{equation}\label{eq: lagrangian}
L(\tilde{X}^i, \tilde{U}^i, \tilde{Y}^i, \tilde{Z}^i,\Xi^i, \Gamma^i) = \sum_{t=0}^T \frac{\mu}{2}\left\|\tilde{x}^i_t-\tilde{y}^i_t + \frac{\xi^i_t}{\mu}\right\|^2 + \frac{\nu}{2}\left\|\tilde{u}^i_t-\tilde{z}^i_t + \frac{\gamma^i_t}{\nu}\right\|^2.
\end{equation}
Here $\xi, \gamma$ are the dual variables and $\mu, \nu$ are the penality parameters;

\item Collect $x^{ji}_t, u^{ji}_t$ and $\xi^{ji}_t, \gamma^{ji}_t$ from agents $j\in\calM^i_t$;

\item Update global variables
\begin{equation} \label{eq: global_update}
y^i_t = \frac{1}{|\calM^i_t|+1}\sum_{j\in\calM^i_t\cup \{i\}}\left(x^{ji}_t + \frac{1}{\mu}\xi^{ji}_t\right); 
\hspace{2mm}
z^i_t = \frac{1}{|\calM^i_t|+1}\sum_{j\in\calM^i_t\cup \{i\}}\left(u^{ji}_t + \frac{1}{\nu}\gamma^{ji}_t\right)
\end{equation}
where $x_t^{ii}$ and $u_t^{ii}$ coincide with $x_t^i$ and $u_t^i$;

\item Collect $y^j_t, z^j_t$ from $j\in\calN^i_t$;

\item Update dual variables
\begin{equation}
\xi^i_t \mathrel{+}= \mu(\tilde{x}^i_t - \tilde{y}^i_t); \hspace{2mm}
\gamma^i_t \mathrel{+}= \nu(\tilde{u}^i_t - \tilde{z}^i_t).\label{eq: dual_update}
\end{equation}
\end{enumerate}

Note that the optimization problem in \cref{eq: primal_problem} is of the form in \cref{eq: dyn_opt} but with constraints. Due to the capability of sampling-based optimizers to handle non-differentiable costs, these constraints can be easily incorporated into the cost function via indicator functions $C(\tilde{x}^i_t)=\rho(\mathbbm{1}_{h(x^i_t, x^{ij}_t)\leq 0} + \mathbbm{1}_{g^i(x^i_t)\leq 0})$ with $\rho \gg 0$ taking a large positive value. With the neighborhood setup, only communication between agent $i$ and the set of agents in $\calN^i_t$ is required (which implies communication with agents in $\calM^i_t$). Dual variables $\xi$ and $\gamma$ have the same dimensionality as the augmented state and control  variables with $\xi^i_t = [\{\xi^{ij}_t\}_{j\in\{i\}\cup \calN^i_t}]$ and $\gamma^i_t = [\{\gamma^{ij}_t\}_{j\in\{i\}\cup \calN^i_t}]$.

\subsection{Algorithm}
\label{sec: algorithm}
\begin{algorithm}[t]
\caption{Distributed Sampling-based MPC}
\begin{algorithmic}[1]
\STATE \textbf{Given:} $x_0$: initial state; $N$: number of agents; $T$: number of MPC steps; $L$: number of ADMM iterations per MPC step;

\STATE Initialize $\Xi_0, \Gamma_0, X_0, U_0, Y_0, Z_0, \calN_0, \calM_0$
\FOR{$t=0$ to $T$}
\STATE $\tilde{x}_t^i = [x^i_t;\{x^{ij}_t\}_{j\in\calN^i_t}], \ \forall i=1,\cdots,N$
\STATE $\tilde{Y}_t^i \leftarrow [Y^i_t; \{Y^j_t\}_{j\in\calN^i_t}]; \ \tilde{Z}_t^i \leftarrow [Z^i_t; \{Z^j_t\}_{j\in\calN^i_t}]$
\FOR{$l=1$ to $L$}
\FOR{$i=1$ to $N$ \textit{in parallel}}
\STATE $\tilde{X}_t^i, \tilde{U}_t^i \leftarrow$ $\textit{SamplingOptimizer}(\tilde{x}_t^i, \tilde{Y}_t^i, \tilde{Z}_t^i, \Xi_t^i, \Gamma_t^i)$ \ (\Cref{alg: optimizer})
\STATE Collect $\tilde{X}_t^j, \tilde{U}_t^j$ from neighbor agents $j\in\calM^i_t$
\STATE $Y^i_t, Z^i_t \leftarrow$ $\textit{GlobalUpdate}(\tilde{X}_t^i, \tilde{U}_t^i, \Xi_t^i, \Gamma_t^i, \calM^i_t)$
\STATE Collect $Y^j_t, Z^j_t$ from neighbor agents $j\in\calN^i_t$
\STATE $\tilde{Y}_t^i \leftarrow [Y^i_t; \{Y^j_t\}_{j\in\calN^i_t}]; \ \tilde{Z}_t^i \leftarrow [Z^i_t; \{Z^j_t\}_{j\in\calN^i_t}]$
\STATE $\Xi^i_t, \Gamma^i_t \leftarrow$ $\textit{DualUpdate}(\tilde{X}^i_t, \tilde{U}^i_t, \tilde{Y}^i_t, \tilde{Z}^i_t)$
\ENDFOR
\ENDFOR
\STATE $x_{t+1} \leftarrow$ $F(x_t, u_t)$
\STATE $\tilde{X}_{t+1}, \tilde{U}_{t+1}, Y_{t+1}, Z_{t+1}, \Xi_{t+1}, \Gamma_{t+1} \leftarrow$ $\textit{RecedeHorizon}(\tilde{X}_{t}, \tilde{U}_{t}, Y_{t}, Z_{t}, \Xi_{t}, \Gamma_{t})$
\STATE $\{\calN^i_{t+1}, \calM^i_{t+1}\} \leftarrow $ $\textit{ComputeNeighborhood}(x_{t+1})$
\STATE $\tilde{X}_{t+1}, \tilde{U}_{t+1}, \Xi_{t+1}, \Gamma_{t+1} \leftarrow$ $\textit{LocalUpdate}(\tilde{X}_{t+1}, \tilde{U}_{t+1}, Y_{t+1}, Z_{t+1})$
\STATE $Y_{t+1}, Z_{t+1} \leftarrow$ $\textit{GlobalUpdate}(\tilde{X}_{t+1}, \tilde{U}_{t+1}, \Xi_{t+1}, \Gamma_{t+1}, \calM_{t+1})$
\STATE $\Xi_{t+1}, \Gamma_{t+1} \leftarrow$ $\textit{DualUpdate}(\tilde{X}_{t+1}, \tilde{U}_{t+1}, \tilde{Y}_{t+1}, \tilde{Z}_{t+1})$
\ENDFOR
\end{algorithmic}
\label{alg: dist_MPC}
\end{algorithm}
\begin{algorithm}[t]
\caption{Sampling Optimizer (\acl{SS})}
\begin{algorithmic}[1]
\STATE \textbf{Parameters:} $M$: number of policy samples; $T'$: prediction horizon, $K$: number of optimization iterations; $\mu, \nu, \rho$: penalty coefficients for state, control consensus and crashing;
\STATE \textbf{Input:} $\tilde{x}_t, \tilde{Y}_t, \tilde{Z}_t, \Xi_t, \Gamma_t$;
\STATE Initialize $\Theta_t^0, \tilde{x}^m_t=\tilde{x}_t, \forall m=1,\cdots,M$
\FOR{$k=0$ to $K-1$}
\FOR{$m=1$ to $M$ \textit{in parallel}}
\FOR{$t'=0$ to $T'-1$}
\STATE $\tilde{u}^m_{t+t'} \sim q(\tilde{u}_{t+t'}; \theta_{t+t'})$
\STATE $\tilde{x}^m_{t+t'+1} = F(\tilde{x}^m_{t+t'}, \tilde{u}^m_{t+t'})$
\ENDFOR
\STATE $\mathcal{J}^m=J(X^m_t,U^m_t) + L(\tilde{X}^m_t, \tilde{U}^m_t, \tilde{Y}_t, \tilde{Z}_t,\Xi_t, \Gamma_t) + C(\tilde{X}^m_t)$
\ENDFOR
\STATE $\Theta^{k+1}_t = \Theta^k_t + \frac{\sum_{m} S(-\mathcal{J}^m) (U^m - \frac{1}{M}\sum_{n} U^{n}))}{\sum_{n}S(-\mathcal{J}^{n})}$
\ENDFOR
\STATE $\tilde{X}_t, \tilde{U}_t \leftarrow$ $\textit{TestPolicy}(\tilde{x}_t, \Theta^k_t)$
\RETURN $\tilde{X}_t, \tilde{U}_t$
\end{algorithmic}
\label{alg: optimizer}
\end{algorithm}

With the framework proposed in \Cref{sec: dist_sampling_derivation}, we introduce the distributed sampling-based MPC algorithm, summarized in Algorithm \ref{alg: dist_MPC} with the procedure for sampling-based optimizers detailed in Algorithm \ref{alg: optimizer}.

Given the initial state, the neighborhood sets $\calN_0, \calM_0$ are computed. The state and control trajectories $X, U$ for all agents are initialized along with their global counterparts such that $X=Y$ and $U=Z$. The dual variables $\Xi, \Gamma$ are initialized as zero. We use subscript on upper case letters to denote trajectories over the planning horizon $T'$ at MPC timstep $t$ (e.g. $X_t = [x_t, x_{t+1}, \cdots, x_{t+T'}]$). At each MPC step, the initial state and global variables are augmented by collecting values from agents in each neighborhood. Subsequently, $L$ \ac{ADMM} iterations are carried out where each agent solves the local optimization problem \eqref{eq: primal_problem} using sampling-based optimizer (\Cref{alg: optimizer}). Each step between 8 and 13 needs to be completed by all agents before carrying out the next one. The agent index is omitted in \Cref{alg: optimizer} for simplicity.

The optimizer solves the local problem by generating $M$ control trajectories from the sampling distribution, propagating state trajectories, and computing a cost for each state-control trajectory. The sampling distribution is updated via a weighted average of the control samples based on the cost. Note that \ac{SS} update law is used but the same steps hold for all perspectives in \Cref{sec: sampling_opt}.

The local solutions are then shared to perform the global and dual variable update \eqref{eq: global_update}\eqref{eq: dual_update}. After optimization, the first control is executed to propagate all agents to the next state and the prediction horizon is receded by 1. The neighborhood sets are recomputed using the new state. The local variables $\tilde{X}^i, \tilde{U}^i$ and dual variables $\Xi, \Gamma$ are updated through the following logic:
\begin{itemize}
\item If agent $j$ stays in the neighborhood, then $\tilde{X}^{ij}, \tilde{U}^{ij}$ and $\Xi^{ij}, \Gamma^{ij}$ remain unchanged;
\item If agent $j$ leaves the neighborhood, then $\tilde{X}^{ij}, \tilde{U}^{ij}$ and $\Xi^{ij}, \Gamma^{ij}$ are dropped;
\item If agent $j$ enters the neighborhood, then $\tilde{X}^{ij}=Y^j, \tilde{U}^{ij}=Z^j$ and $\Xi^{ij}, \Gamma^{ij}$ are set to zero.
\end{itemize}
Finally, the global and dual variables are updated using the new local variables.

\section{Simulation Results}
\label{sec: results}

\bgroup
\begin{table*}[t]
    \centering
    \caption{Mean cost and variance comparison for a point mass system over 20 seeds.}
    \begin{tabular}{@{} l  | c | c | c @{}}
    \toprule
     \multicolumn{1}{c}{} & \multicolumn{1}{c}{CEM} & \multicolumn{1}{c}{MPPI} & \multicolumn{1}{c}{Tsallis} \\
    \midrule
    Mean & $2.734 \times 10^2$ & $3.29\times 10^2$ & $\bf{2.725\times 10^2}$ \\
    \midrule
    Variance &  $\bf{5.27}$ & $39.3$ & $7.64$ \\
    \bottomrule
    \end{tabular}
    \label{tab: numerical_data}
\end{table*}
\egroup

\begin{table}[t]
    \centering
    \caption{Comparison of different distributed framework against the centralized framework for different tasks and policy distribution over 5 random seeds. The success rate measures whether collision occurs during the trial.}
    {\fontsize{7}{9}\selectfont 
    \begin{tabular}{
        @{}
        l
        c
        S[table-format=1.2e-2, table-number-alignment = center]
        S[table-format=3.0,table-number-alignment = center,table-figures-decimal=1,table-auto-round]
        S[table-format=1.2e-2, table-number-alignment = center]
        S[table-format=3.0,table-number-alignment = center,table-figures-decimal=1,table-auto-round]
        S[table-format=1.2e-2, table-number-alignment = center]
        S[table-format=3.0,table-number-alignment = center,table-figures-decimal=1,table-auto-round]
        S[table-format=1.2e-2, table-number-alignment = center]
        S[table-format=3.0,table-number-alignment = center,table-figures-decimal=1,table-auto-round]
        @{}
    }
    \toprule
    & & \multicolumn{2}{c}{MPPI} & \multicolumn{2}{c}{CEM}
    & \multicolumn{2}{c}{Tsallis} & \multicolumn{2}{c}{Central} \\
    \cmidrule(lr){3-4} \cmidrule(lr){5-6} \cmidrule(lr){7-8} \cmidrule(lr){9-10}
    {System} & {Policy} & {Mean} & {Success} & {Mean} & {Success} & {Mean} & {Success} &
    {Mean} &
    {Success} \\
    \midrule
    \multirow{3.8}{*}{\begin{tabular}[c]{@{}c@{}}Dubins\\Narrow\\Crossing\end{tabular}}
    & \begin{tabular}[c]{@{}c@{}} UG \\ \end{tabular} &
    1.66e3 & 100 & 5.97e2 & 100 & 6.19e2 & 100 & 6.36e2 & 100 \\
    \cmidrule(l){2-10} 
    & \begin{tabular}[c]{@{}c@{}}GM \\ \end{tabular} &
    1.58e3 & 100 & 5.76e2 & 100 & 5.59e2 & 100 & 6.77e2 & 100 \\
    \cmidrule(l){2-10} 
    & S & 2.57e3 & 80 & 9.25e2 & 100 & 8.49e2 & 100 & 9.43e2 & 100\\
    \midrule
    \multirow{3.8}{*}{\begin{tabular}[c]{@{}c@{}}Dubins\\Swap\end{tabular}}
    & \begin{tabular}[c]{@{}c@{}}UG \\\end{tabular} &
    2.49e5 & 100 & 1.25e5 & 100 & 1.17e5 & 100 & 2.22e5 & 100 \\
    \cmidrule(l){2-10} 
    & \begin{tabular}[c]{@{}c@{}}GM \\\end{tabular} &
    2.53e5 & 100 & 1.18e5 & 100 & 1.15e5 & 100 & 2.21e5 & 100 \\
    \cmidrule(l){2-10} 
    & S &
    5.59e5 & 40 & 2.82e5 & 100 & 5.12e5 & 100 & 4.61e5 & 100 \\
    \midrule
    \multirow{3.8}{*}{\begin{tabular}[c]{@{}c@{}}Dubins\\Formation\end{tabular}}
    & \begin{tabular}[c]{@{}c@{}}UG \\\end{tabular} &
        n/a & 0 & 2.31e5 & 100 & 2.37e5 & 100 & 6.26e5 & 60 \\
    \cmidrule(l){2-10} 
    & \begin{tabular}[c]{@{}c@{}}GM \\\end{tabular} &
        n/a & 0 & 2.29e5 & 100 & 2.28e5 & 100 & 6.39e5 & 60\\
    \cmidrule(l){2-10} 
    & S & n/a & 0 & 2.43e5 & 80 & 2.44e5 & 60 & n/a & 0\\
    \midrule
    \multirow{3.8}{*}{\begin{tabular}[c]{@{}c@{}}Quadcopter\\Formation\end{tabular}}
    & \begin{tabular}[c]{@{}c@{}}UG\\ \end{tabular} &
        1.91e6 & 80 & 6.40e5 & 100 & 6.42e5 & 100 & 9.12e5 & 100 \\
    \cmidrule(l){2-10} 
    & \begin{tabular}[c]{@{}c@{}}GM \\ \end{tabular} &
        2.48e6 & 80 & 6.33e5 & 100 & 6.32e5 & 100 & 8.8e5 & 100 \\
    \cmidrule(l){2-10} 
    & S &
        6.90e7 & 100 & 6.55e5 & 100 & 6.68e5 & 80 & 8.99e5 & 100 \\
    \bottomrule
    \end{tabular}
    }
    \label{tab:data_table_new}
\end{table}
In this section, we test the proposed frameworks using sampling optimizers derived from the three different perspectives, while also comparing against the equivalent centralized approachexs as the baseline. We denote the \ac{VO} approach as \ac{MPPI}, \ac{VI} approach as Tsallis, and \ac{SS} approach as \ac{CEM}. All comparisons are performed on 4 tasks in simulation: narrow crossing, swapping, formation with Dubins vehicle and quadcopter formation. To ensure fair comparisons, all hyperparameters for each method are tuned using a combination of the TPE algorithm \citep{bergstra2011algorithms} from the Neural Network Intelligence (NNI) AutoML framework \citep{Microsoft_Neural_Network_Intelligence_2021} and hand tuning. \Cref{tab:data_table_new} compares the resulting mean costs and success rate (no collision) over 5 seeds for the three different distributed optimizers and the best performing centralized optimizer\footnote{A video comparison of the different frameworks using unimodal Gaussian policy can be found \href{https://youtu.be/XHDJhNDqwtU}{here}.}.

\subsection{Numerical Example}

We first test the cost and standard deviation of the different frameworks on a simple single-agent dynamic optimization problem. The task is to control a point mass starting from an initial position to reach a target location while avoiding an obstacle. The point mass system is characterized by the deterministic double integrator dynamics. Each framework is tuned with the NNI AutoML tool on 1 seed and tested on 20 random seeds. The results are included in \Cref{tab: numerical_data}. We can observe that while Tsallis has the lowest mean cost, \ac{CEM} results in the lowest cost standard deviation. The results are aligned with the analysis in \cite{wang2021variational}, where \ac{CEM} is infinitely risk averse and Tsallis can be both risk averse and risk seeking based on the choice of $r$. Intuitively, the Tsallis framework assigns higher weights on the best performing samples, leading to better cost mean but worse standard deviation.

\begin{figure}[t]
  \begin{subfigure}{0.245\textwidth}
    \includegraphics[width=\linewidth]{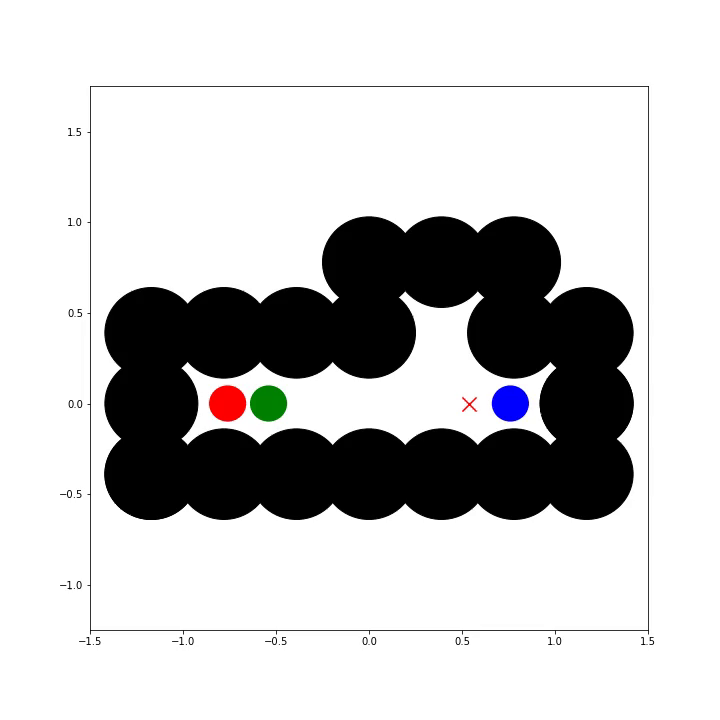}
    \caption{$t=0$} \label{fig: dubins_t0}
  \end{subfigure}%
  \begin{subfigure}{0.245\textwidth}
    \includegraphics[width=\linewidth]{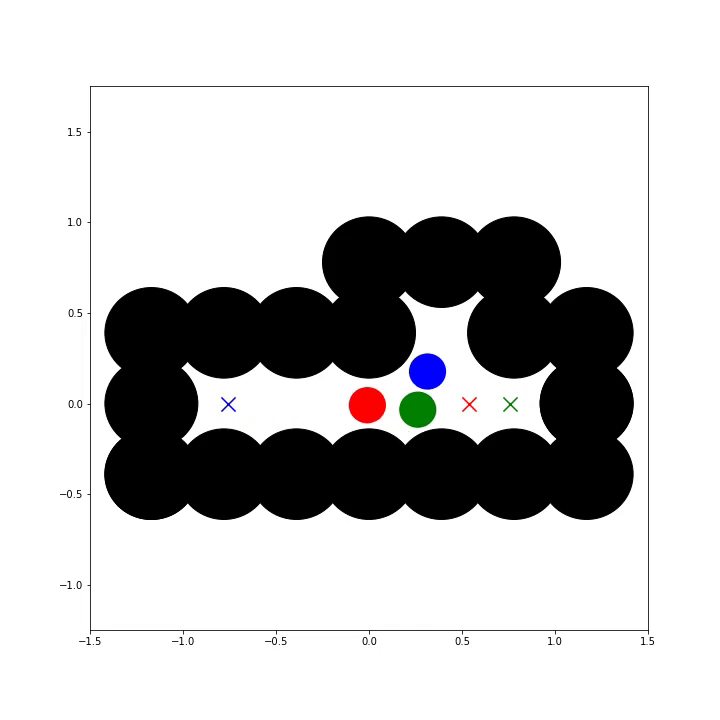}
    \caption{$t=50$} \label{fig: dubins_t50}
  \end{subfigure}%
  \begin{subfigure}{0.245\textwidth}
    \includegraphics[width=\linewidth]{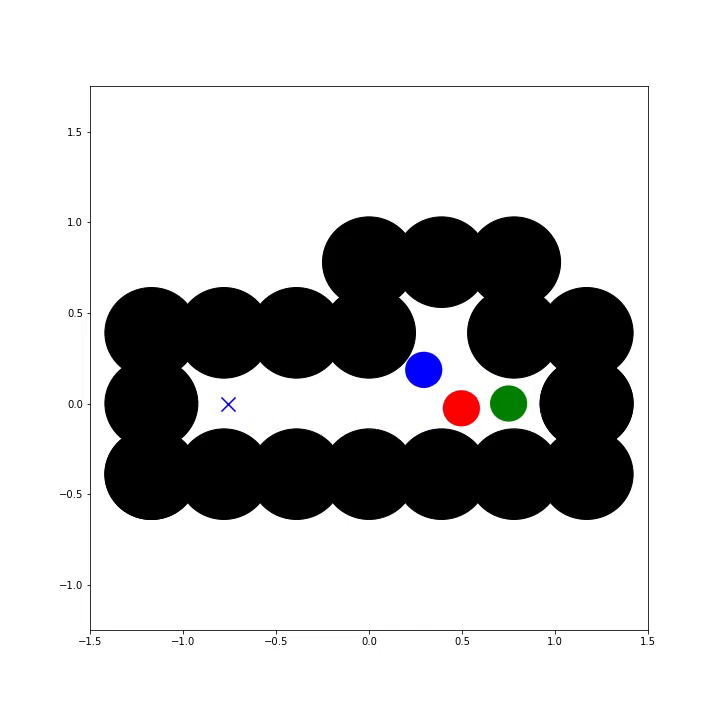}
    \caption{$t=180$} \label{fig: dubins_t180}
  \end{subfigure}
  \begin{subfigure}{0.245\textwidth}
    \includegraphics[width=\linewidth]{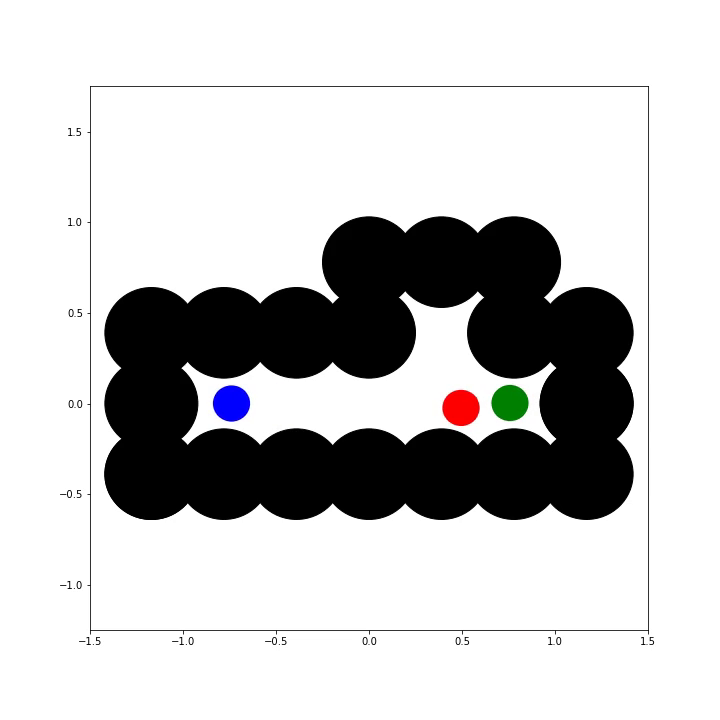}
    \caption{$t=400$} \label{fig: dubins_t400}
  \end{subfigure}
  \caption{Snapshots of the 3-agent ``narrow crossing'' task with distributed Tsallis. Targets are marked by the cross of the same color as the agents.}
  \label{fig: dubins_narrow_passage}
  \vspace{-10pt}
\end{figure}

\begin{figure}[t]
  \begin{subfigure}{0.25\textwidth}
    \includegraphics[width=\linewidth]{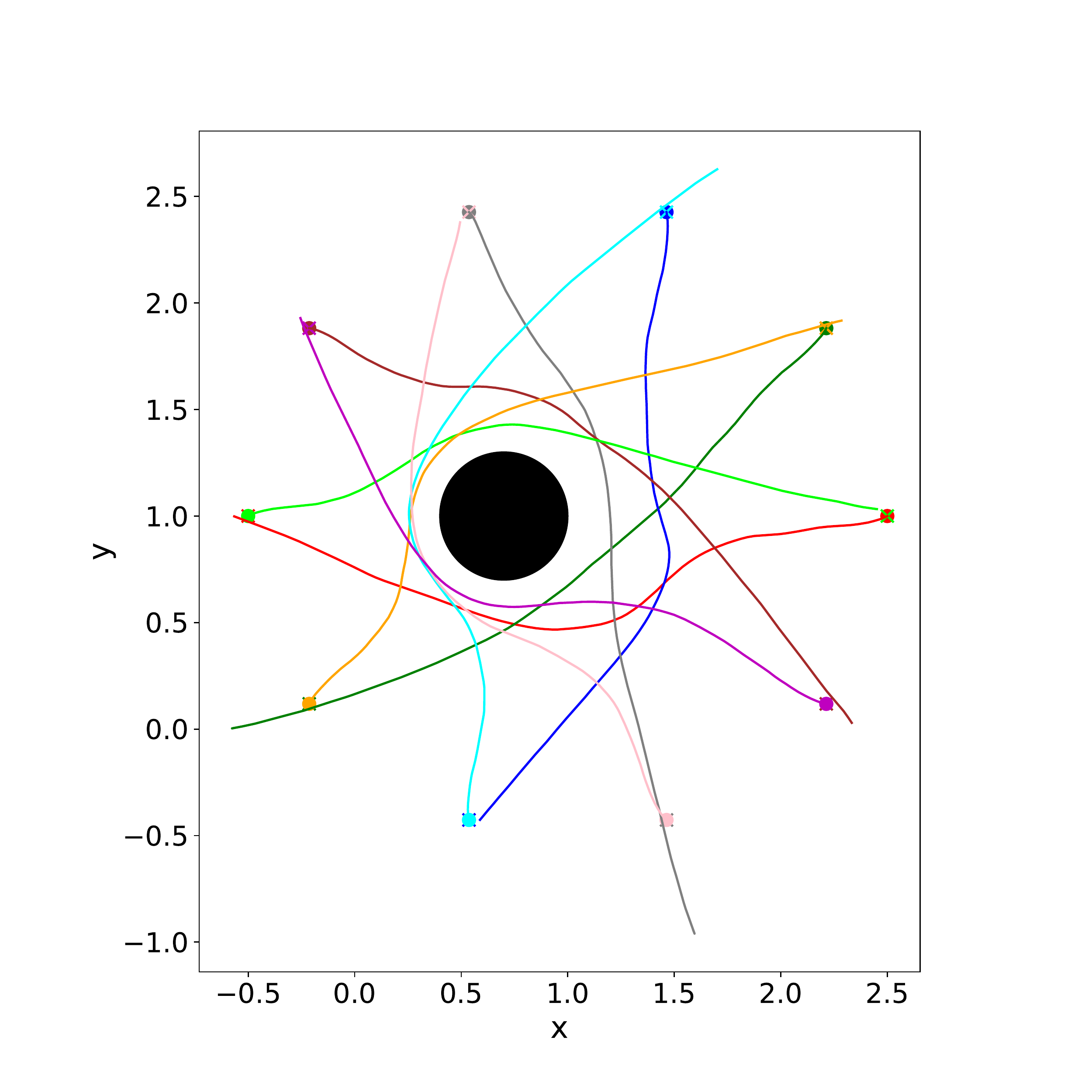}
    \caption{Centralized} \label{fig: dubins_swap_baseline}
  \end{subfigure}%
  \begin{subfigure}{0.245\textwidth}
    \includegraphics[width=\linewidth]{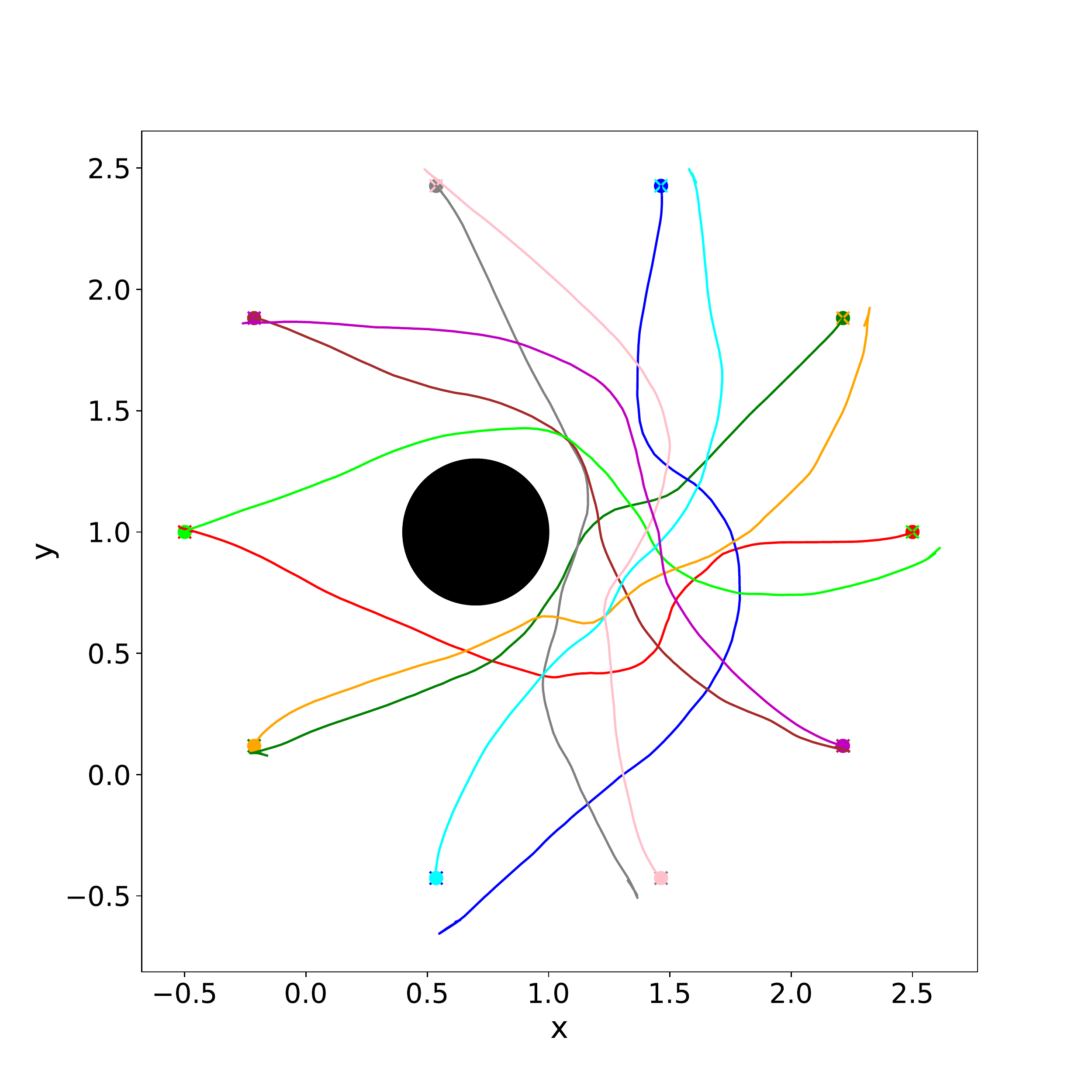}
    \caption{CEM} \label{fig: dubins_swap_cem}
  \end{subfigure}%
  \begin{subfigure}{0.245\textwidth}
    \includegraphics[width=\linewidth]{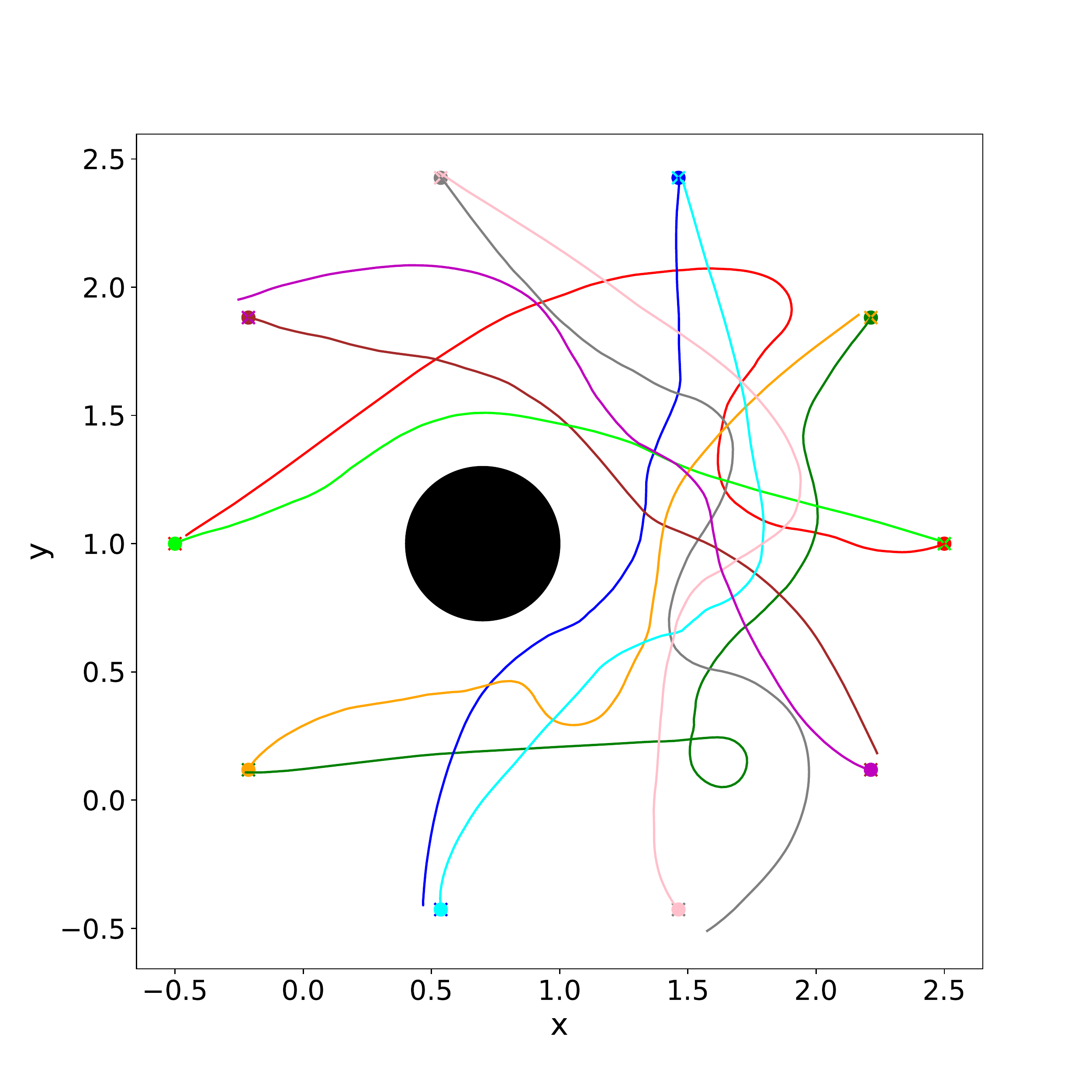}
    \caption{MPPI} \label{fig: dubins_swap_mppi}
  \end{subfigure}
  \begin{subfigure}{0.245\textwidth}
    \includegraphics[width=\linewidth]{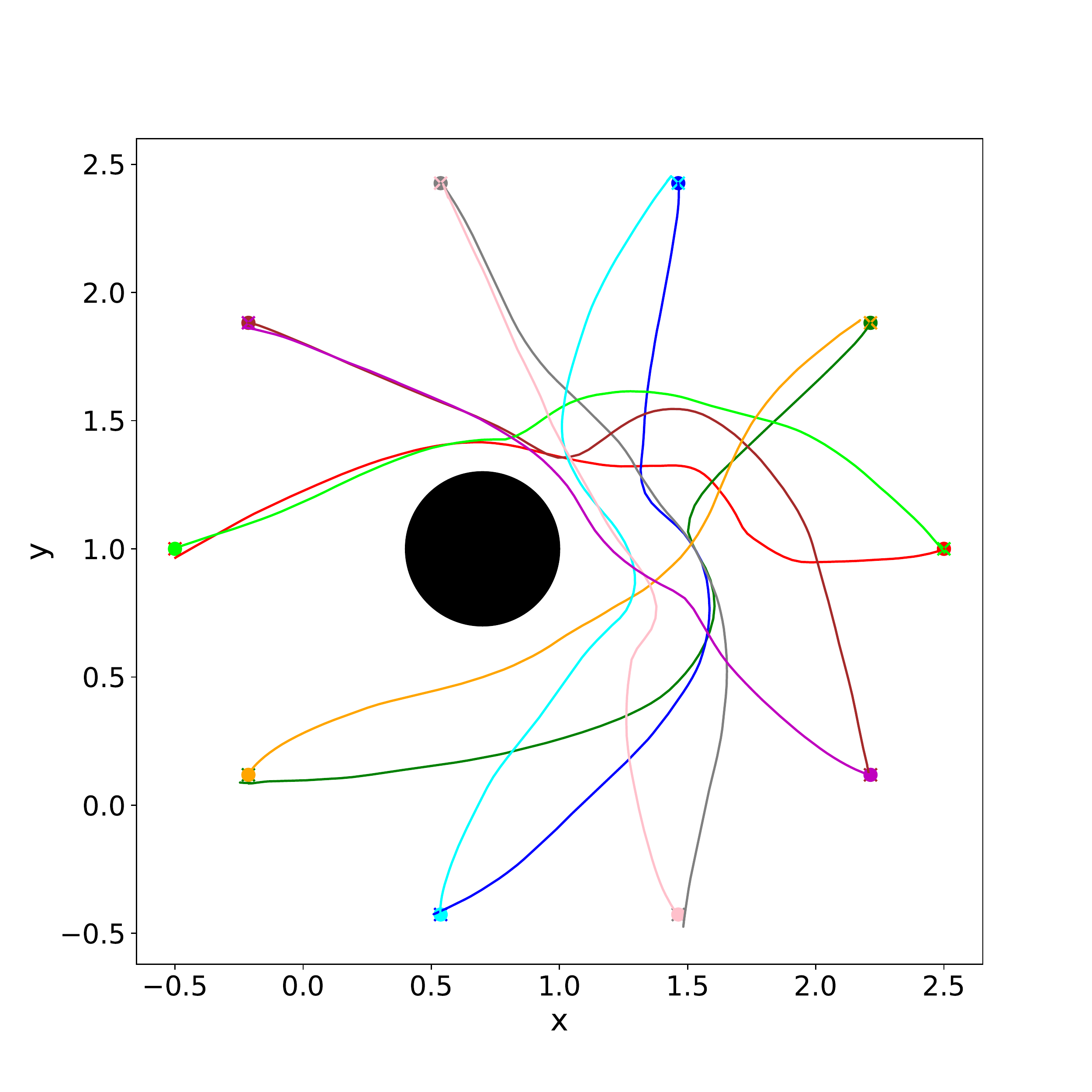}
    \caption{Tsallis} \label{fig: dubins_swap_tsallis}
  \end{subfigure}
  \caption{Trajectory comparison of the 10-agent swapping task with Dubins vehicle.}
  \label{fig: dubins_swap}
  \vspace{-10pt}
\end{figure}

\subsection{Dubins Narrow Crossing}

Subsequently, the performance of the proposed methods is tested on a more challenging task where three vehicles are in a ``narrow crossing'' whose width is such that only one vehicle can fit vertically (\Cref{fig: dubins_narrow_passage}). The only way for the vehicles to reach to their desired targets is to take advantage of a gap, which of course requires proper coordination. The resulting costs for all methods are provided in \Cref{tab:data_table_new}. In \Cref{fig: dubins_narrow_passage}, snapshots from the motion of the agents using distributed Tsallis optimizer are provided.

\begin{wrapfigure}{r}{0.45\textwidth}
\vspace{-30 pt}
    \centering
    \includegraphics[width=0.45\textwidth]{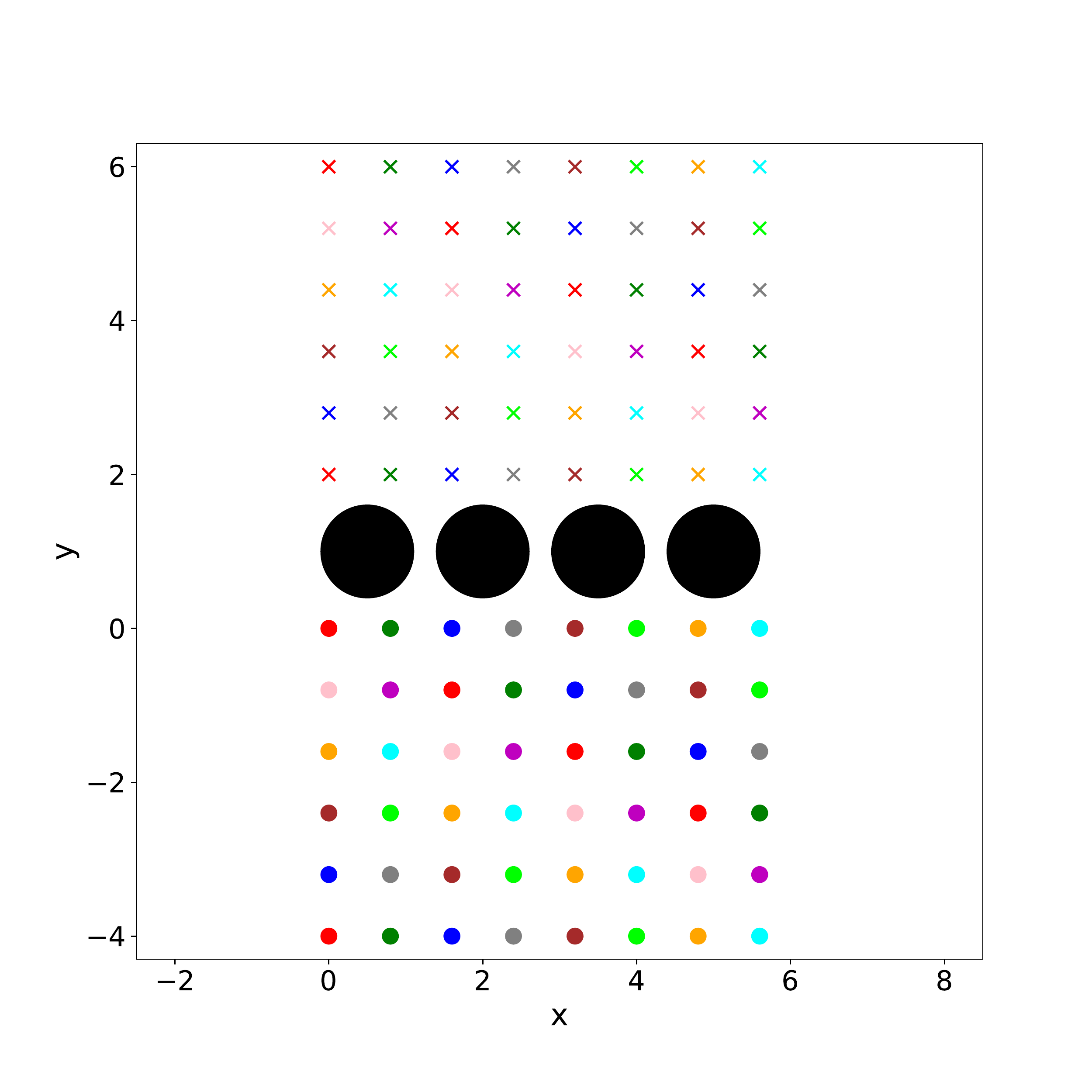}
    \vspace*{-10 mm}
    \caption{Dubins vehicle formation task with 48 agents.}
    \vspace{-20 pt}
    \label{fig:dubins_formation}
\end{wrapfigure}

\subsection{Dubins Swap}

We also perform the comparison on a 10-agent swapping task, where the agents are arranged in a circle with the goal of reaching the starting point of the agent directly across while avoiding an obstacle, as shown in \Cref{fig: dubins_swap}. The obstacle is placed off center so the optimal behavior is not circling around the obstacle as it would be in the centered obstacle case. Figures \ref{fig: dubins_swap_baseline} through \ref{fig: dubins_swap_tsallis} show that the centralized approach takes the circling maneuver, whereas the agents go through the side with more room in the distributed schemes. 


\subsection{Dubins Formation}

The distributed frameworks are then tested on a 48-agent formation task. As shown in \Cref{fig:dubins_formation}, the agents need to move through the obstacles to reach the desired formation. To increase the task difficulty, the openings between obstacles are such that only one agent can pass through at a time. 



\begin{wrapfigure}{r}{0.45\textwidth} 
\vspace{-20 pt}
    \centering
    \includegraphics[trim=70 120 70 180, clip,width=0.45\textwidth]{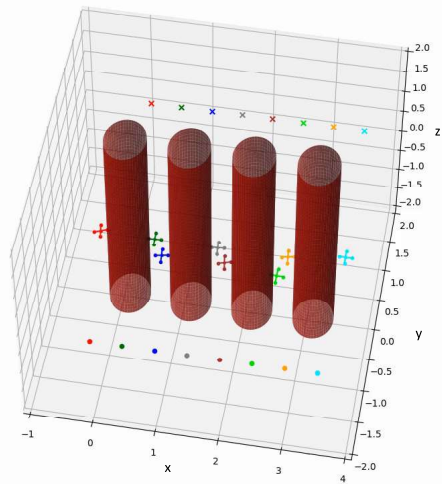}
    \vspace*{-10 mm}
    \caption{Snapshot of quadcopters formation task with distributed CEM.}
    \vspace{-10 pt}
    \label{fig:quads_formation}
\end{wrapfigure} 

\subsection{Quadcopter Formation}

We also perform a formation experiment with the more challenging quadcopter dynamics. We use the quadcopter dynamics derived in \cite{Sabatino2015QuadrotorCM}, in which the control inputs are the thrust and the three dimensional torques. The goal for the quadcopters is to travel from some initial starting points to assigned targets while avoiding collisions with each other and additional cylindrical obstacles. The obstacles are placed such that no two agents can pass together at the same elevation. The quadcopters' thrusts are also penalized such that it is not easy for agents to pass together at different elevations. This is to show how neighboring quadcopters can achieve consensus and pass between the obstacles without colliding. \Cref{fig:quads_formation} shows a snapshot of a successful and conscientious distributed CEM solution of $8$ quadcopters colored differently in which the corresponding starting points are the filled circles and the targets are the crosses. It can be seen that the neighboring agents line up to pass between the obstacles without colliding. 

\begin{figure}[t]
    \centering
    \includegraphics[width=0.9\linewidth]{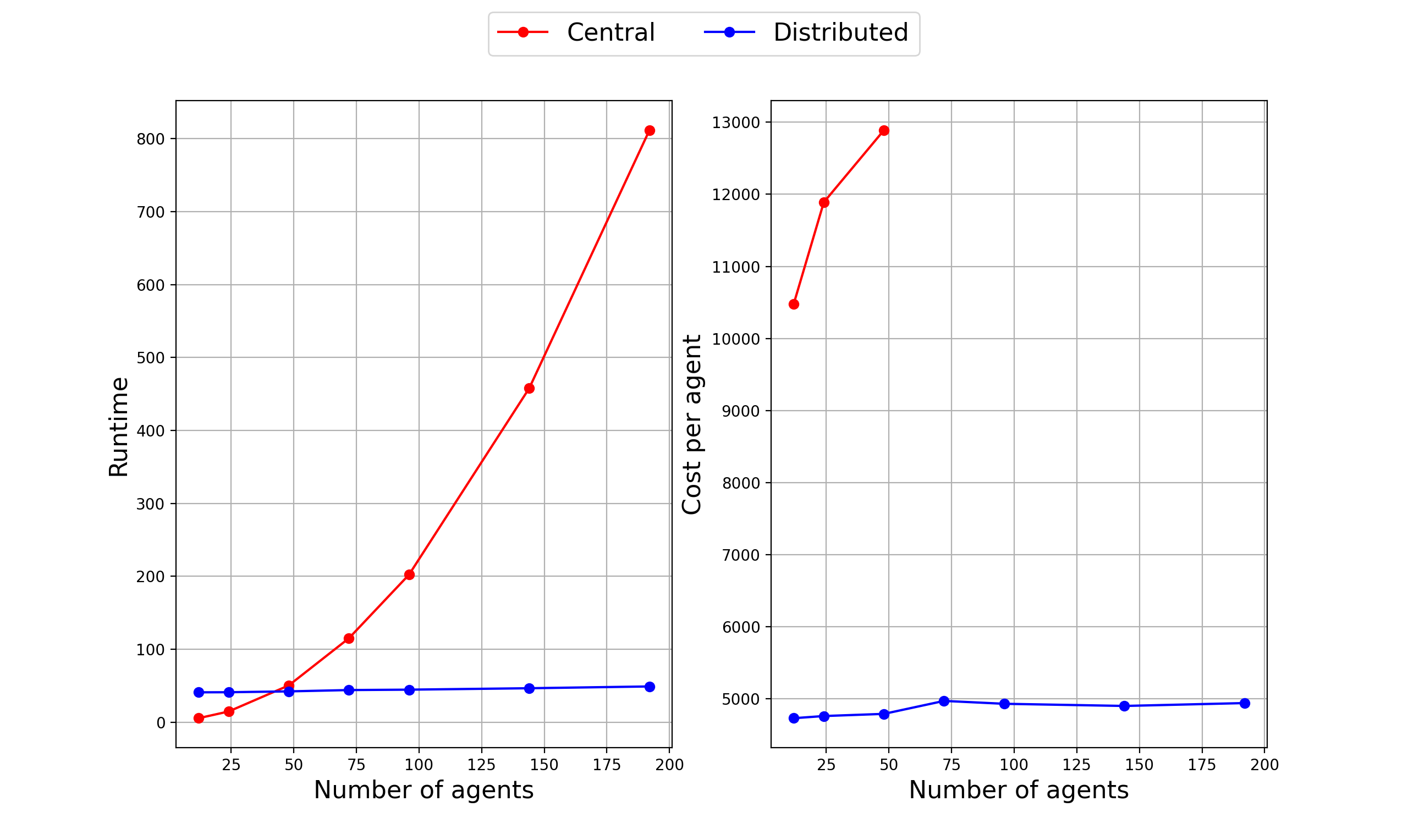}
    \caption{Scaling comparison between the Central and Distributed scheme on the Dubins formation experiment. \textit{Left:} Runtime comparison; \textit{Right:} Cost per agent comparison. The missing data points for the central scheme indicate crashing result.}
    \label{fig:scaling_comparison}
\end{figure}

\subsection{Scaling}

We test the scalability of the frameworks on 7 different Dubins formation experiments with 12, 24, 48, 72, 96, 144, 196 agents. The runtime and cost per agent are compared between the distributed and centralized schemes when using the Tsallis optimizer and a unimodal Gaussian policy in \Cref{fig:scaling_comparison}. It can be observed that the runtime of the distributed scheme stays almost constant while the runtime of the centralized one grows superlinearly. Similarly, the cost per agent also stays constant for the distributed scheme. For the centralized one, the cost per agent more than doubles that of the distributed scheme for the 12-agent experiment and grows as the number of agent increases. The high cost indicates failure of task completion, as also indicated in \Cref{fig:scaling_comparison}. For experiments with more than 48 agents, the central scheme results in crashing. The scaling experiments demonstrate the distributed scheme's capability of handling large numbers of agents.

\subsection{Discussion}
Several observations can be drawn from the results in \Cref{tab:data_table_new}. First of all, distributed \ac{CEM} or Tsallis always outperforms the centralized framework. The distributed frameworks break the problem into smaller local ones, making them easier to solve. In addition, the quadratic consensus term in \cref{eq: lagrangian} leads to better conditioned problems. The second observation is that \ac{CEM} and Tsallis have similar levels of performance and consistently achieve lower mean costs than \ac{MPPI}. This is because the cost transform of \ac{MPPI} assigns either too much weight on a single best performing sample or non-negligible weights on worst performing samples. In theory, when both frameworks are sufficiently tuned, Tsallis should perform equally well or better than \ac{CEM} since it is a generalization. In practice, we indeed observe comparable results between the two frameworks. 

Finally, the multi-modal Gaussian mixture and Stein variational policies tend to suffer from mode collapse in all experiment settings due to the presence of many agents and obstacles. Gaussian mixture policy collapses to a single mode, leading to similar performance to that of the unimodal Gaussian policy. 

\section{Conclusions}
\label{sec: conclusions}
In this paper, we systematically review three different perspectives on sampling-based dynamic optimization and provide theoretical justification to heuristics in prior works. We provide a unifying theoretical analysis on the convergence and sample complexity of the perspectives. We then introduce a distributed framework for scaling general sampling-based dynamic optimizers for multi-agent control using consensus ADMM. The proposed framework splits the multi-agent problem into smaller local problems with respect to their neighbors, resulting in a distributed MPC algorithm that can adapt the neighborhood at every timestep. The framework is tested on 4 different tasks in simulation, including a series of scaling experiments up to 192 agents for which a centralized approach inflates the computational complexity due to the number of samples needed. We compare between 3 different distributed optimizers with 3 different policy distributions and against their centralized versions.


\acks{We would like to thank professor Enlu Zhou for the helpful discussions.}


\newpage

\appendix
\section*{Appendix}
\label{app:theorem}



\section{Derivations for Different Sampling-based Optimization Perspective}
\subsection{KL VO}
\label{appendix: KL_VO}

To derive the optimal Gibbs distribution, we formulate the Lagrangian from the objective as

\begin{equation}
\calL = \mathbb{E}_{q(\tau)}[J(X)] + \lambda \KL{q(\tau)}{p(\tau)} + \beta (\int q(\tau)\rd U - 1),
\end{equation}
where the last term ensures that the distribution integrates to 1. Taking the gradient of the Lagrangian and setting it to zero we get 
\begin{align}
\nabla \calL &= J(X) + \lambda \ln q(\tau) + \lambda - \lambda \ln p(\tau) + \beta = 0\\
\Rightarrow q^*(\tau) &= \exp\left(-\frac{1}{\lambda} J(X) \right)\exp(-1-\beta)p(\tau)\\
\Rightarrow q^*(U) &= \exp\left(-\frac{1}{\lambda} J(X) \right)\exp(-1-\beta)p(U).
\label{eq: KL VO qstar}
\end{align}
Integrating both sides and equating them to 1 we get
\begin{align}
\int q^*(U) \rd U &= \exp(-1-\beta)\int \exp\left(-\frac{1}{\lambda} J(X) \right) p(U)\rd U = 1\\
\Rightarrow \exp(-1-\beta) &= \frac{1}{\Eb_{p(U)}\left[\exp\left(-\frac{1}{\lambda}J(X)\right)\right]}.
\label{eq: KL VO exp}
\end{align}
Plugging \cref{eq: KL VO exp} into \cref{eq: KL VO qstar}, we get the optimal distribution to be
\begin{equation} \label{eq: gibbs_dist}
q^* = \frac{\exp\left(-\frac{1}{\lambda}J(X)\right)p(U)}{\Eb_{p(U)}\left[\exp\left(-\frac{1}{\lambda}J(X)\right)\right]}.
\end{equation}
Nevertheless, we cannot directly sample from the optimal distribution. To derive an iterative update law, we minimize the KL divergence between the controlled and optimal distributions. 
For each of the following update laws, we make use of the following equality:
\begin{align} \label{eq:derive:kl_vo}
    q^{k+1}(U)
    &= \argmin_{q(U;\Theta)} \KL{q^*(U)}{q(U;\Theta)} \\
    &= \argmin_{q(U;\Theta)} \int q^*(U) \log q^*(U) \rd U - \int q^*(U) \log q(U;\Theta) \rd U .
\end{align}
We can drop the first term as it doesn't relate to the $\argmin_{q(U;\Theta)}$, which gives us 
\begin{align}
    q^{k+1} &= \argmin_{q(U;\Theta)}  - \int q^*(U) \log q(U;\Theta) \rd U \\
    &= \argmax_{q(U;\Theta)} \int \hat{q}^*(U) \log q(U;\Theta) \rd U \\ 
    &= \argmax_{q(U;\Theta)} \int \frac{\exp\left(-\frac{1}{\lambda}J(X)\right)p(U)}{\Eb_{p(U)}\left[\exp\left(-\frac{1}{\lambda}J(X)\right)\right]} \log q(U;\Theta) \rd U\\
    &= \argmax_{q(U;\Theta)}  \frac{\Eb_{q(U;\Theta)}\left[\frac{p(U)}{q(U;\Theta)}\exp\left(-\frac{1}{\lambda}J(X)\right)\log q(U;\Theta)\right]}{\Eb_{q(U;\Theta)}\left[\frac{p(U)}{q(U;\Theta)}\exp\left(-\frac{1}{\lambda}J(X)\right)\right]}.
\end{align}
For a Gaussian distribution with mean $\Theta$, the update law can be derived by setting the gradient of the objective with respect to $\Theta$ to be zero as follows
\begin{equation}
\Theta^{k+1} = \frac{\Eb_{q^k(U;\Theta)}\left[\exp \left(-\frac{1}{\lambda}J(\tau)\right) U \right]}{\Eb_{q^k(U;\Theta)}\left[ \exp \left(-\frac{1}{\lambda}J\right)\right]},
\end{equation}
where $J(\tau)=J(X) + J(U)$ with $J(U)=\frac{\lambda}{2}(\Theta\T\Sigma \Theta + 2U\T\Sigma\Theta)$ being the control cost resulting from the importance sampling.

\subsection{Renyi VO}
\label{appendix: Renyi_VO}

To derive the Renyi VO, we start with the optimal Gibbs distribution in \cref{eq: gibbs_dist}. Minimizing the Renyi divergence between the optimal distribution and the controlled distribution

\begin{align}
q^{k+1}&=\argmin_{q(U;\Theta)}\D{\alpha}{q^{*}(U)}{q(U;\Theta)}\\
&= \argmin_{q(U;\Theta)}\frac{1}{\alpha-1}\log \int \bigg( \frac{q^*(U)}{q(U;\Theta)} \bigg)^{\alpha-1} q^*(U) \rd U \\
&=\argmin_{q(U;\Theta)}\frac{1}{\alpha-1}\log \int \bigg( \frac{q^*(U) p(U)}{p(U) q(U;\Theta)} \bigg)^{\alpha-1}  q^*(U) \rd U \\
&= \argmin_{q(U;\Theta)} \frac{1}{\alpha-1}\log \int \bigg( \frac{q^*(U)}{ p(U)}\bigg)^{\alpha-1} \bigg(\frac{p(U)}{ q(U;\Theta)} \bigg)^{\alpha-1}  q^*(U)\rd U \\
& =\argmin_{q(U;\Theta)} \frac{1}{\alpha-1}\log \int \bigg( \frac{q^*(U)}{ p(U)}\bigg)^{\alpha} \bigg(\frac{p(U)}{ q(U;\Theta)} \bigg)^{\alpha-1}  p(U)\rd U\\
& =\argmin_{q(U;\Theta)} \frac{1}{\alpha-1}\log \int \bigg( \frac{q^*(U)}{ p(U)}\bigg)^{\alpha} \bigg(\frac{p(U)}{ q(U;\Theta)} \bigg)^\alpha  q(U;\Theta)\rd U.
\end{align}
Similar to KL VO, for a Gaussian distribution with mean $\Theta$, the update law can be derived by setting the gradient of the objective with respect to $\Theta$ to be zero as follows
\begin{equation}
\Theta^{k+1}  =  \frac{\Eb_{q^k(U;\Theta)}\left[ \left( \exp\left(-\frac{1}{\lambda} J(\tau)\right) \right)^{\alpha} U \right]}{\Eb_{q^k(U;\Theta)}\left[\left( \exp\left(-\frac{1}{\lambda} J(\tau)\right) \right)^{\alpha} \right]},
\end{equation}
where $J(\tau)$ is defined as in \Cref{appendix: KL_VO}.

\subsection{Tsallis VI}
\label{appendix: Tsallis_VI}

The optimal distribution can be derived similarly as in the VO approach through the Lagrangian
\begin{equation}
\calL = \lambda^{-1} \Eb_{q(U)}[J(\tau)] + \D{r}{q(U;\Theta)}{p(U)} + \beta(\int q(U)\rd U - 1).
\end{equation}
Setting the gradient to zero, we have
\begin{align}
\nabla \calL
&= \lambda^{-1} J + \frac{ r}{r-1}\left(\frac{q(U)}{p(U)}\right)^{r-1} - \beta
=0 \\
\Rightarrow q^*
&= \left(
    \frac{r-1}{ r}\right)^{\frac{1}{r-1}}
    \left(\beta - \lambda^{-1} J \right)^{\frac{1}{r-1}} p(U) .
\end{align}
Redefining $\lambda^{-1} = \alpha(r-1) \lambda^{-1}$ and integrating both sides to 1, we get
\begin{equation}
q^* = \frac{\exp_r\left(-\frac{1}{\lambda}J(\tau)\right)p(U)}{\Eb_{p(U)}\left[\exp_r\left(-\frac{1}{\lambda}J(\tau)\right)\right]}.
\end{equation}
Since $p(U)$ is an arbitrary prior distribution, as opposed to the uncontrolled distribution by definition in the VO approach, the update law can be derived by setting the prior distribution as the current controlled distribution $q^k$:
\begin{equation}
q^{k+1} = \frac{\exp_r\left(-\frac{1}{\lambda}J(\tau)\right)q^k(U;\Theta)}{\Eb_{q^k(U;\Theta)}\left[\exp_r\left(-\frac{1}{\lambda}J(\tau)\right)\right]}.
\end{equation}
With this the parameter update for the Gaussian distribution can be computed by taking the expectation on both sides
\begin{equation}
\Theta^{k+1} = \frac{\Eb_{q^k(U;\Theta)}\left[\exp_r\left(-\frac{1}{\lambda}J(\tau)\right)U\right]}{\Eb_{q^k(U;\Theta)}\left[\exp_r\left(-\frac{1}{\lambda}J(\tau)\right)\right]}.
\end{equation}

\subsection{SS}
\label{appendix: SS}

The stochastic search scheme transforms the original objective using a non-decreasing shape function for the algorithmic development while preserving the set of solutions. In addition, a monotonic logarithmic transform is performed on the objective to ensure that the update law is scale-free. With the assumption of the control distribution from the exponential family, the gradient can be taken with respect to the natural parameters at each time step t:
\begin{equation}
\begin{split}
\nabla_{\eta_t}  \ln\left(\mathbb{E}\left[S\left(- J(\tau)\right)\right]\right) &= \frac{\int S(-J(\tau))\nabla_{\eta_t} q(U;\Theta)\mathrm{d}U}{\int S(-J(\tau))q(U;\Theta)\mathrm{d}U} \\
&= \frac{\int S(-J(\tau))q(U;\Theta) \nabla_{\eta_t} \ln q(U;\Theta)\mathrm{d}U}{\int S(-J(\tau))q(U;\Theta)\mathrm{d}U}.
\end{split}
\end{equation}
The term $\nabla_{\eta_t} \ln q(U;\Theta)$ can be calculated as:
\begin{equation}
\begin{split}
\nabla_{\eta_t} \ln q(U;\Theta) &= \nabla_{\eta_t} \Big(\sum_{t=0}^{T-1}\big(\ln h(u_t) + \eta_t^\mathrm{T}T(u_t)-A(\theta_t) \big)\Big) \\
&= T(u_t) - \Eb[T(u_t)].
\end{split}
\end{equation}
Plugging it back into the gradient we have:
\begin{equation}
\nabla_{\eta_t} l(\theta) = \frac{\int S(-J(\tau))q(U;\Theta) (T(u_t) - \Eb[T(u_t)]) \mathrm{d}U}{\int S(-J(\tau))q(U;\Theta)\mathrm{d}U}\label{eq:gradient}.
\end{equation}
With the gradient the update rule for $\theta$ can be found as:
\begin{equation} \label{eq: ss_update}
\eta_t^{k+1} 
= \eta_t^k + \alpha^k\Bigg(\frac{\mathbb{E}\Big[S(-J(\tau)) (T(u_t) - \Eb[T(u_t)]) \Big]}{\mathbb{E} \Big[S(-J(\tau))\Big]}\Bigg).
\end{equation}
Plugging in the parameters of the Gaussian distribution with fixed variance, we get
\begin{equation}
\begin{split}
h(u_t) &= \frac{1}{\sqrt{(2\pi)^n}|\Sigma|}\exp\Big(-\frac{1}{2}u_t^\mathrm{T}\Sigma^{-1}u_t\Big), \hspace{2mm} A(\theta_t) = \frac{1}{2} \theta_t^\mathrm{T}\Sigma^{-1}\theta_t, \\
T(u_t) &= \Sigma^{-\frac{1}{2}} u_t, \hspace{2mm} \eta_t(\theta_t) = \Sigma^{-\frac{1}{2}}\theta_t.
\end{split}
\end{equation}
The parameter update can be found as:

\begin{equation}
\theta_t^{k+1}
= \theta_t^k + \frac{\mathbb{E}\Big[S(-J(\tau)) (u_t - \theta_t^k) \Big]}{\mathbb{E} \Big[S(-J(\tau))\Big]}.
\end{equation}

\section{Derivations of Different Policies}
\label{appendix: policy}

For each of the following update laws, we make use of the following equality:
\begin{align} \label{eq:derive:kl_mle}
    q^{k+1}(U)
    &= \argmin_{q} \KL{\hat{q}^*(U)}{q(U)} \\
    &= \argmin_{q} \int \hat{q}^*(U) \log \hat{q}^*(U) \rd U - \int \hat{q}^*(U) \log q(U) \rd U .
\end{align}
We can drop the first term as it doesn't relate to the $\argmin_{q}$:
\begin{align}
    &= \argmin_{q}  - \int \hat{q}^*(U) \log q(U) \rd U \\ 
    &= \argmax_{q} \int \hat{q}^*(U) \log q(U) \rd U \\ 
    &= \argmax_{q} \sum_{m=1}^M w^{(m)} \log q(U^{(m)}) .
\end{align}

\subsection{Unimodal Gaussian}
For the case of a policy class of unimodal Gaussian distributions where the controls $u_t \in \Rb^{n_u}$ of each timestep are independent, the p.d.f. for each timestep's control has the form
\begin{equation}
    \mathcal{N}(u_t; \mu_t, \Sigma_t)
    = \frac{1}{\sqrt{(2 \pi)^{n_u} \abs{\Sigma_t}}} \exp \left( -\frac{1}{2} (u_t - \mu_t)\T \Sigma_t^{-1} (u_t - \mu_t) \right),
\end{equation}
where we have that 
\begin{align}
    \nabla_{\mu_t} \log \mathcal{N}(u_t; \mu_t, \Sigma_t)
        &= (u_t - \mu_t)\T \Sigma_t^{-1} \label{eq:derive_update:unimodal:mu} \\
    \nabla_{\Sigma_t^{-1}} \log \mathcal{N}(u_t; \mu_t, \Sigma_t)
        &= \frac{1}{2} \Sigma_t - \frac{1}{2} (u_t - \mu_t)(u_t - \mu_t)\T \label{eq:derive_update:unimodal:sigma} .
\end{align}
Minimizing the KL divergence between $\pi(U)$ and $\tilde{q}^*(U)$ and using \Cref{eq:derive:kl_mle}, we obtain
\begin{align}
    \pi^{k+1}(U)
    &= \argmin_{\pi(U) \in \Pi} \KL{\tilde{q}^*(U)}{\pi(U)} \\
    &= \argmax_{\mu, \Sigma} \sum_{m=1}^M \sum_{t=0}^{T-1} w^{(m)} \log \mathcal{N}(U^{(m)}; \mu_t, \Sigma_t).
\end{align}
Hence, we have that
\begin{align}
     0 &= \sum_{m=1}^M w^{(m)} \nabla_{\mu_t} \log \mathcal{N}(U^{(m)}; \mu_t, \Sigma_t) \\
     0 &= \sum_{m=1}^M w^{(m)} \nabla_{\Sigma_t} \log \mathcal{N}(u_t^{(m)}; \mu_t, \Sigma_t).
\end{align}
Finally, using \Cref{eq:derive_update:unimodal:mu,eq:derive_update:unimodal:sigma} then results in the following update laws
\begin{align}
    \mu_t &= \sum_{m=1}^M w^{(m)} u_t^{(m)} \\
    \Sigma_t &= \sum_{m=1}^M w^{(m)} (u_t^{(m)} - \mu_t) (u_t^{(m)} - \mu_t)\T .
\end{align}

\subsection{Gaussian Mixutre}
In this case, we consider a Gaussian mixture model $\pi(U; \theta)$ with $L$ components
with parameters $\theta \coloneqq \{ \theta_l \}_{l=1}^{L}$ and $\theta_l \coloneqq \{ (\phi_l, \{ \mu_{l, t}, \Sigma_{l, t} \}_{t=0}^{T-1} )$ such that
\begin{align}
    \pi(U; \theta) &= \prod_{t=0}^{T-1} \pi_t(u_t; \theta) \\
    \pi_t(u_t; \theta) &\coloneqq \sum_{l=1}^L \pi_l \mathcal{N}(u_t; \mu_l, \Sigma_l).
\end{align}
Directly trying to minimize the KL divergence between $\pi(U; \theta)$ and $\tilde{q}^*(U)$ and using \Cref{eq:derive:kl_mle} gives us
\begin{align}
    \theta^*
    &= \argmax_\theta \sum_{m=1}^M \sum_{t=0}^{T-1} \left( w^{(m)} \log \sum_{l=1}^L \phi_l \mathcal{N}(u_t^{(m)}; \mu_{l, t}, \Sigma_{l, t}) \right). \label{eq:gmm_em_argmax}
\end{align}
However, \Cref{eq:gmm_em_argmax} does not admit a closed form solution. 

Instead, we can use the expectation maximization algorithm and introduce the latent variables $\{ Z_m \}_{m=1}^M$, $Z_m \in [L]$ which form a categorical distribution $q(Z)$ over each one of the $L$ mixtures and indicate the mixture each sample $U^{(m)}$ was sampled from.
Then, we can write $\log \pi(U; \theta)$ as
\begin{align}
    \log \pi(U; \theta) &= \int q(z) \log \pi(U; \theta) \rd z \\
    &= \int q(z) \log \frac{ \pi(U; \theta)p(z|U; \theta) }{p(z | U; \theta)} \rd z \\
    &= \int q(z) \log \frac{ p(U, z; \theta)}{p(z | U; \theta)} \rd z \\
    &= \int q(z) \log \frac{ p(U, z; \theta)}{q(z)} \rd z - \int q(z) \log \frac{p(z|U; \theta)}{q(z)} \rd z \\
    &= F(q, \theta) + \KL{q}{p} .
\end{align}
where $F(q, \theta)$ is the ELBO, and provides a lower bound for the log likelihood $\log p(U; \theta)$.
Now, to optimize $F(q, \theta)$, we estimate $q(z)$ via $p(z | U; \theta^{k})$, where $\theta^{k}$ is the previous iteration's parameters.
Then, the ELBO becomes
\begin{align}
    F(q, \theta)
    &= \int q(z) \log \frac{p(U, z; \theta)}{q(z)} \rd z \\
    &= \int q(z) \log p(U, z; \theta) \rd z - \int q(z) \log q(z) \rd z \\
    &= \int p(z | U; \theta^{k}) \log p(U, z; \theta) \rd z - \int p(z | U; \theta^k) \log p(z | U; \theta^k) \rd z \\
    &= Q(\theta, \theta^k) + H(z | U) .
\end{align}
Since the second term is a function of $\theta^k$ and thus is not dependent on $\theta$, it suffices to optimize $Q(\theta, \theta^k)$.
To do so, the EM algorithm consists of two steps:
\begin{enumerate}
    \item E-step: Compute $p(z | U; \theta^k)$.
    \item M-step: Compute $\argmax_\theta Q(\theta, \theta^k)$.
\end{enumerate}
In our case, we have that
\begin{align}
    Q(\theta, \theta^k)
    &= \int p(z | U; \theta^{k}) \log p(U, z; \theta) \rd z \\
    &= \ExP{z \sim p(z | U; \theta^k) }{ \sum_{m=1}^M \sum_{l=1}^L \sum_{t=0}^{T-1} \mathbbm{1}\{Z_n = l\} w^{(m)} \left( \log \phi_l + \log \mathcal{N}(u_t^{(m)} | \mu_{l, t}, \Sigma_{l, t}) \right) } \\
    &= \sum_{m=1}^M \sum_{l=1}^L \sum_{t=0}^{T-1} p(Z_n = l | u_t^{(m)}; \theta^k) w^{(m)} \left( \log \phi_l + \log \mathcal{N}(u_t^{(m)} | \mu_{l, t}, \Sigma_{l, t}) \right) \\
    &= \sum_{m=1}^M \sum_{l=1}^L \sum_{t=0}^{T-1} \eta_l(u_t^{(m)}) w^{(m)} \left( \log \phi_l + \log \mathcal{N}(u_t^{(m)} | \mu_{l, t}, \Sigma_{l, t}) \right).
\end{align}
where we have defined $\eta_l(u_t^{(m)}) \coloneqq p(Z_n = l | u_t^{(m)}; \theta^{k})$ for simplicity.
To compute this, note that
\begin{align}
    p(Z_k = l| u_t^{(m)}; \theta^k)
    &= \frac{p(u | Z_n = l; \theta^k) p(Z_n = l; \theta^{k})}{\sum_{l'=1}^L p(u | Z_n = l'; \theta^{k}) p(Z_n = l'; \theta^{k})} \\
    &= \frac{ \phi_l \mathcal{N}(u_t^{(m)}; \mu_l^{k}, \Sigma_l^{k}) }{\sum_{l'=1}^L \phi_l \mathcal{N}(u^{(m)}; \mu_{l'}^{k}, \Sigma_{l'}^{k}) } .
\end{align}
To solve for the optimal $\theta^*$ in the M-step, we apply the first order conditions to $\mu_{l, t}$ and $\Sigma_{l, t}$ to obtain
\begin{align}
    0 &= \sum_{m=1}^M \eta_l(u_t^{(m)}) w^{(m)} (u_t^{(m)} - \mu_{l, t}^{k+1}) \Sigma_{l, t}^{-1} \\
    \implies \mu_{l, t}^{k+1} &= \frac{1}{N_{l, t}} \sum_{m=1}^M \eta_l(u_t^{(m)}) w^{(m)} u_t^{(m)},
\end{align}
and
\begin{align}
    0 &= \sum_{m=1}^M \eta_l(u_t^{(m)}) w^{(m)} \left( \frac{1}{2} \Sigma_{l, t}^{k+1} - \frac{1}{2} (u_{l, t} - \mu_{l, t}^{k+1})(u_{l, t} - \mu_{l, t}^{k+1} )\T \right) \\
    \implies \Sigma_{l, t}^{k+1} &= \frac{1}{N_{l, t}} \sum_{m=1}^M \eta_l(u_t^{(m)}) w^{(m)} (u_t^{(m)} - \mu_{l, t}^{k+1}) (u_t^{(m)} - \mu_{l, t}^{k+1})\T , 
\end{align}
where we have defined $N_{l, t} \coloneqq \sum_{m=1}^M \eta_l(u_t^{(m)}) w^{(m)}$ for convenience.
Since $\phi_l^*$ includes the constraint $\sum_{l=1}^L \phi_l^* = 1$, we solve for $\phi_l^*$ by applying the Lagrangian multiplier $\lambda$, yielding the following Lagrangian:
\begin{equation}
    L_t =
    \sum_{m=1}^M \sum_{l=1}^L \sum_{t=0}^{T-1} \eta_l(u_t^{(m)}) w^{(m)} \big( \log \phi_l + \log \mathcal{N}(u_t^{(m)} | \mu_{l, t}, \Sigma_{l, t} \big) + \lambda \left(1 - \sum_{l=1}^L \phi_l \right) .
\end{equation}
Applying first order conditions for $\phi_l$ gives
\begin{align}
    0 &= \sum_{m=1}^M \sum_{t=0}^{T-1} \eta_l(u_t^{(m)}) w^{(m)} \frac{1}{\phi_l^{k+1}} - \lambda \\
    \implies \phi_l^{k+1} &= \frac{1}{\lambda} \sum_{m=1}^M \sum_{t=0}^{T-1} \eta_l(u_t^{(m)}) w^{(m)} \\
                   &= \frac{1}{\lambda} \sum_{t=0}^{T-1} N_{l, t} .
\end{align}
To solve for the Lagrange multiplier $\lambda$, we plug in the constraint to get
\begin{equation}
    \lambda = \sum_{l=1}^L \sum_{m=1}^M \sum_{t=0}^{T-1} \eta_l(u_t^{(m)}) w^{(m)} = \sum_{l=1}^L \sum_{t=1}^{T-1} N_{l, t} .
\end{equation}
Hence, we finally get that 
\begin{equation}
    \phi_l = \frac{N_l}{\sum_{l'=1}^L N_{l'}},
\end{equation}
where we have defined $N_l \coloneqq \sum_{t=0}^{T-1} N_{l, t}$.

\subsection{Stein Variational Policy}
Finally, we derive the update law for the choice of a fully non-parametric policy via SVGD.
SVGD solves the variational inference problem
\begin{equation}
    \pi^* = \argmin_{\pi \in \Pi} \{ \KL{\pi(U)}{q^*(U)} \} ,
\end{equation}
over the set $\Pi \coloneqq \{ z | z = T(x) \}$ consisting of all distributions obtained by smooth transforms $T$ of random variables $x$, where $x$ is drawn from some tractable reference distribution.

One can show that the direction of steepest descent $\phi^*$ which maximizes the negative gradient $-\nabla_\epsilon \KL{ q_{[T]} }{ p } \vert_{\epsilon=0}$ in zero-centered balls in the RKHS $\mathcal{H}^d$ has the form
\begin{align}
    \phi^*(\cdot) = \Eb_{U \sim \pi} \left[ k(U, \cdot) \nabla_U \log q^*(U) + \nabla_U k(U, \cdot) \right] .
\end{align}
Hence, one can then compute the optimal distribution $\pi^*$ by iteratively applying $\phi^*$ to some initial distribution $p_0$.
However, given that we only have an empirical approximation $\tilde{q}^*$ of the optimal distribution $q^*$, the gradient of $\log q^*$ cannot be easily computed.
To solve this, following \cite{lambert2020stein}, instead of optimizing directly over the distribution of the \textit{controls} $U$,
one can instead optimize over some distribution $g$ of \textit{parameters} $\theta$ of a parametrized policy $\hat{\pi}(U; \theta)$, such that
\begin{equation}
    \pi(U) = \Eb_{\theta \sim g} \left[\hat{\pi}_\theta(U) \right] \quad \text{and} \quad q^*(\theta) = \Eb_{U \sim \hat{\pi}(\cdot; \theta)} \left[ q^*(U) \right] .
\end{equation}
By doing so, we obtain the following new \ac{VI} problem and steepest descent direction respectively:
\begin{align}
    g^* &= \argmin_{\theta \in \Theta} \{ \KL{ g(\theta) }{ q^*(\theta) } \} \\
    \hat{\phi}^*(\cdot) &= \Eb_{\theta \sim g} \Big[ \hat{k}(\theta, \cdot) \nabla_\theta \log \Eb_{U \sim \hat{\pi}(\cdot; \theta)} [q^*(U)] + \nabla_\theta \hat{k}(\theta, \cdot) \Big] ,
\end{align}
As a result, we can now take the gradient of $\log \Eb_{U \sim \hat{\pi}(\cdot; \theta)} [q^*(U)]$ as
\begin{align}
    \nabla_\theta \log \Eb_{U \sim \hat{\pi}(\cdot; \theta)} [q^*(U)]
    &= \frac{\nabla_\theta \Eb_{U \sim \hat{\pi}(\cdot; \theta)} [q^*(U)]}{\Eb_{U \sim \hat{\pi}(\cdot; \theta)} [q^*(U)]} \\
    &= \frac{\int q^*(U) \nabla_\theta \hat{\pi}(U; \theta) \, \rd U }{\Eb_{U \sim \hat{\pi}(\cdot; \theta)} [q^*(U)]} \\
    &= \frac{\Eb_{U \sim \hat{\pi}(\cdot; \theta)} [q^*(U) \nabla_\theta \log \hat{\pi}(U; \theta)]}{\Eb_{U \sim \hat{\pi}(\cdot; \theta)} [q^*(U)]} .
\end{align}
Discretizing the above since we only have a discrete approximation $\tilde{q}^*$ of $q^*$, we get that
\begin{equation}
    \nabla_\theta \log \Eb_{U \sim \hat{\pi}(\cdot; \theta)} [q^*(U)]
    = \frac{\sum_{m=1}^M \tilde{q}^*(U^{(m)}) \nabla_\theta \log \hat{\pi}(U^{(m)}; \theta) }{\sum_{m=1}^M \tilde{q}^*(U^{(m)}) } .
\end{equation}
Hence, by letting $g$ be an empirical distribution $g(\theta) = \sum_{l=1}^L \bm{1}_{\{\theta_l\}}$, we obtain the following update rule for the particles $\{\theta_l\}_{l=1}^L$:
\begin{equation}
    \theta_l^{k+1} = \theta_l^k + \hat{\phi}^*(\theta_l^k) .
\end{equation}

\vskip 0.2in
\bibliography{reference}

\end{document}